\documentclass[a4paper,12pt]{article}

\usepackage[T1]{fontenc}
\usepackage[utf8]{inputenc}
\usepackage{graphicx}

\usepackage{amsmath}
\usepackage{amsfonts}
\usepackage{amssymb}
\usepackage{amsthm}
\usepackage{array}
\usepackage{color}
\usepackage{enumerate}

\newtheorem{theorem}{Theorem}[section]

\newtheorem{definition}[theorem]{Definition}
\newtheorem{proposition}[theorem]{Proposition}
\newtheorem{lemma}[theorem]{Lemma}
\newtheorem{corollary}[theorem]{Corollary}

\newtheorem{remark}[theorem]{Remark}

\numberwithin{equation}{section}

\newcommand{\Z}{\mathbb{Z}}

\newcommand{\R}{\mathbb{R}}
\newcommand{\C}{\mathbb{C}}

\newcommand{\B}{\mathbb{B}}
\newcommand{\Ebb}{\mathbb{E}}
\newcommand{\Lbb}{\mathbb{L}}

\newcommand{\Sbb}{\mathbb{S}}
\newcommand{\Tbb}{\mathbb{T}}

\newcommand{\Acal}{\mathcal{A}}
\newcommand{\Bcal}{\mathcal{B}}
\newcommand{\Ccal}{\mathcal{C}}

\newcommand{\Fcal}{\mathcal{F}}

\newcommand{\Hcal}{\mathcal{H}}
\newcommand{\Ical}{\mathcal{I}}

\newcommand{\Mcal}{\mathcal{M}}

\newcommand{\Pcal}{\mathcal{P}}

\newcommand{\norm}[2]{\left\| #1 \right\|_{#2}}
\newcommand{\dd}{\;{\rm d}}
\newcommand{\de}{{\rm d}}

\DeclareMathOperator{\id}{id}

\DeclareMathOperator{\Leb}{Leb}
\DeclareMathOperator{\Liouv}{Liouv}

\DeclareMathOperator{\Jac}{Jac}

\DeclareMathOperator{\GL}{GL}
\DeclareMathOperator{\SL}{SL}
\DeclareMathOperator{\Supp}{Supp}

\DeclareMathOperator{\Card}{Card}
\DeclareMathOperator{\Cov}{Cov}

\DeclareMathOperator{\Vol}{Vol}
\DeclareMathOperator{\VectSpan}{span}

\usepackage[top=3cm, bottom=2cm, left=1.5cm, right=1.5cm]{geometry}

\title{Keplerian shear in ergodic theory}
\author{Damien THOMINE}
\date{}

\begin{document}

\maketitle

\begin{abstract}

Many integrable physical systems exhibit Keplerian shear. We look at this phenomenon 
from the point of view of ergodic theory, where it can be seen as mixing conditionally 
to an invariant $\sigma$-algebra. In this context, we give a sufficient criterion 
for Keplerian shear to appear in a system, investigate its genericity and, 
in a few cases, its speed. Some additional, non-Hamiltonian, examples are discussed.
\end{abstract}

When a celestial body is orbiting circularily around another, Kepler's third law 
asserts that the period of the orbit is proportional to the radius of the orbit 
at the power $3/2$: closer bodies complete their orbits faster. When one considers 
bodies whose size is non-negligible with respect to the radius of the orbit, 
this difference of orbital periods induces a shearing effect, called Keplerian shear~\cite{Tiscareno:2012}. 
Kelperian shear is most notable in planetary rings, for instance Saturn's. As a consequence, 
any large-scale heterogeneity of the rings is wrapped around the rings, until -- 
for large enough times -- it equidistributes radially (see Fig~\ref{fig:Saturne}): 
Keplerian shear explains the radial symmetry of large planetary rings.

\begin{figure}[h!]
\vspace{1cm}
\centering
\includegraphics[scale=0.4]{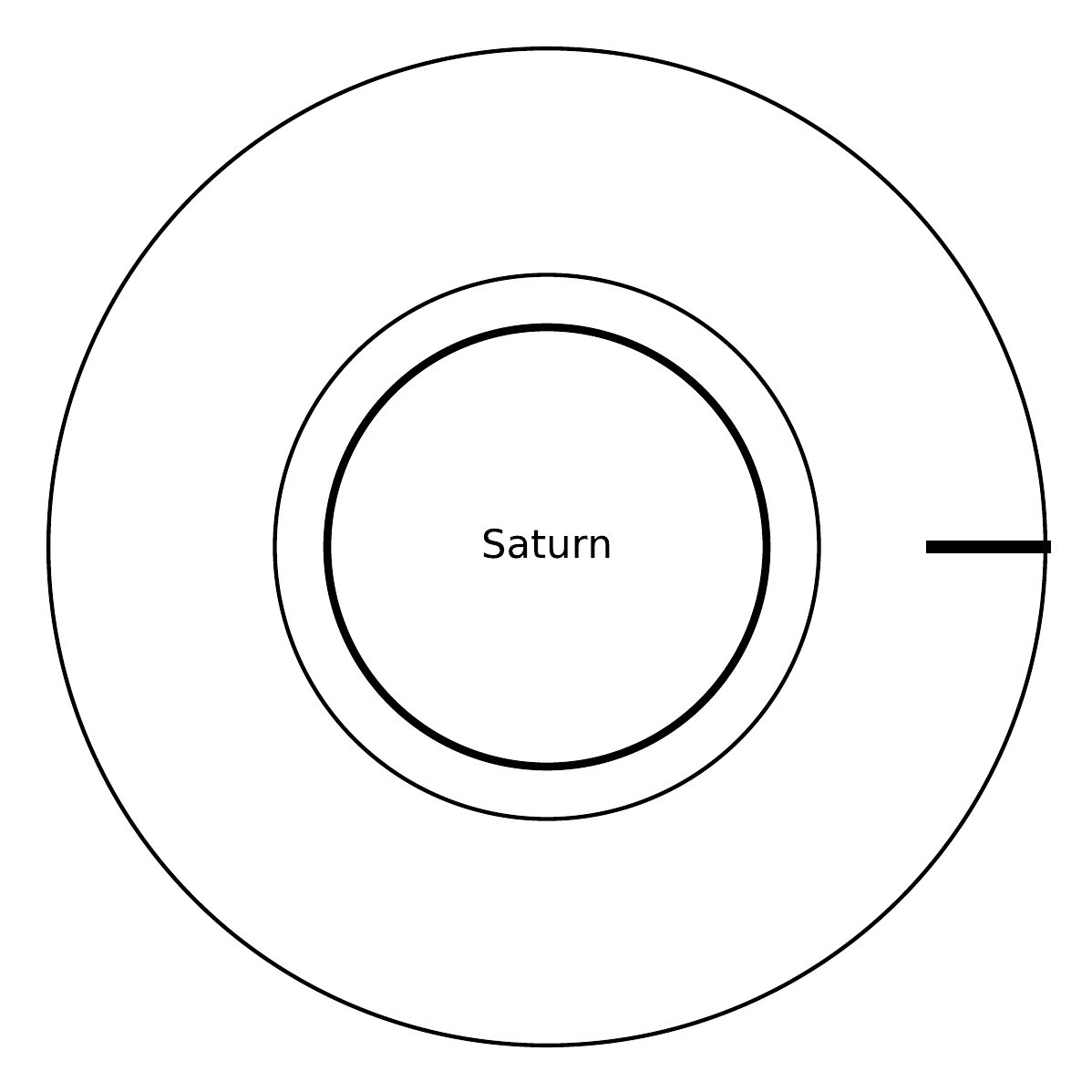}
\includegraphics[scale=0.4]{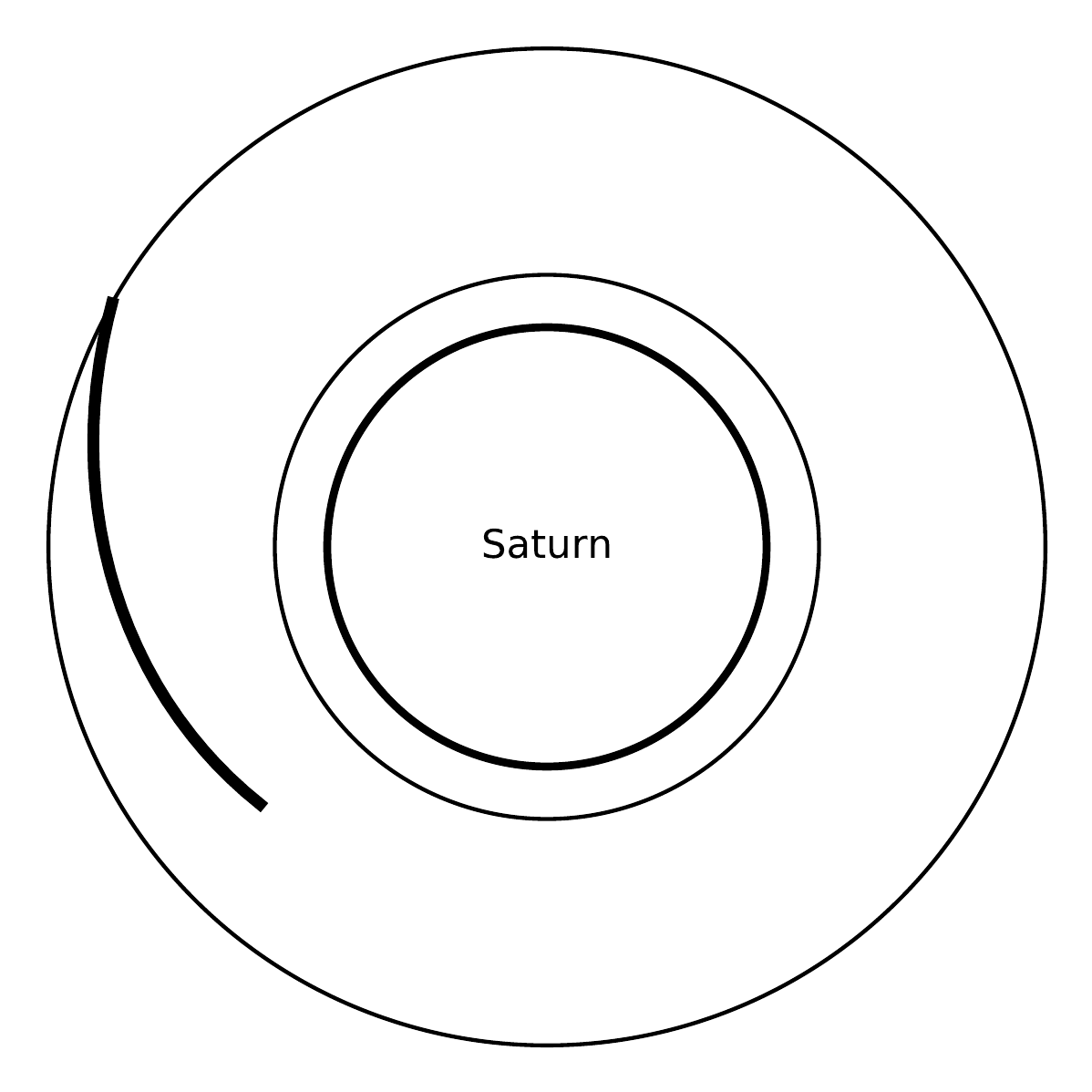}
\includegraphics[scale=0.4]{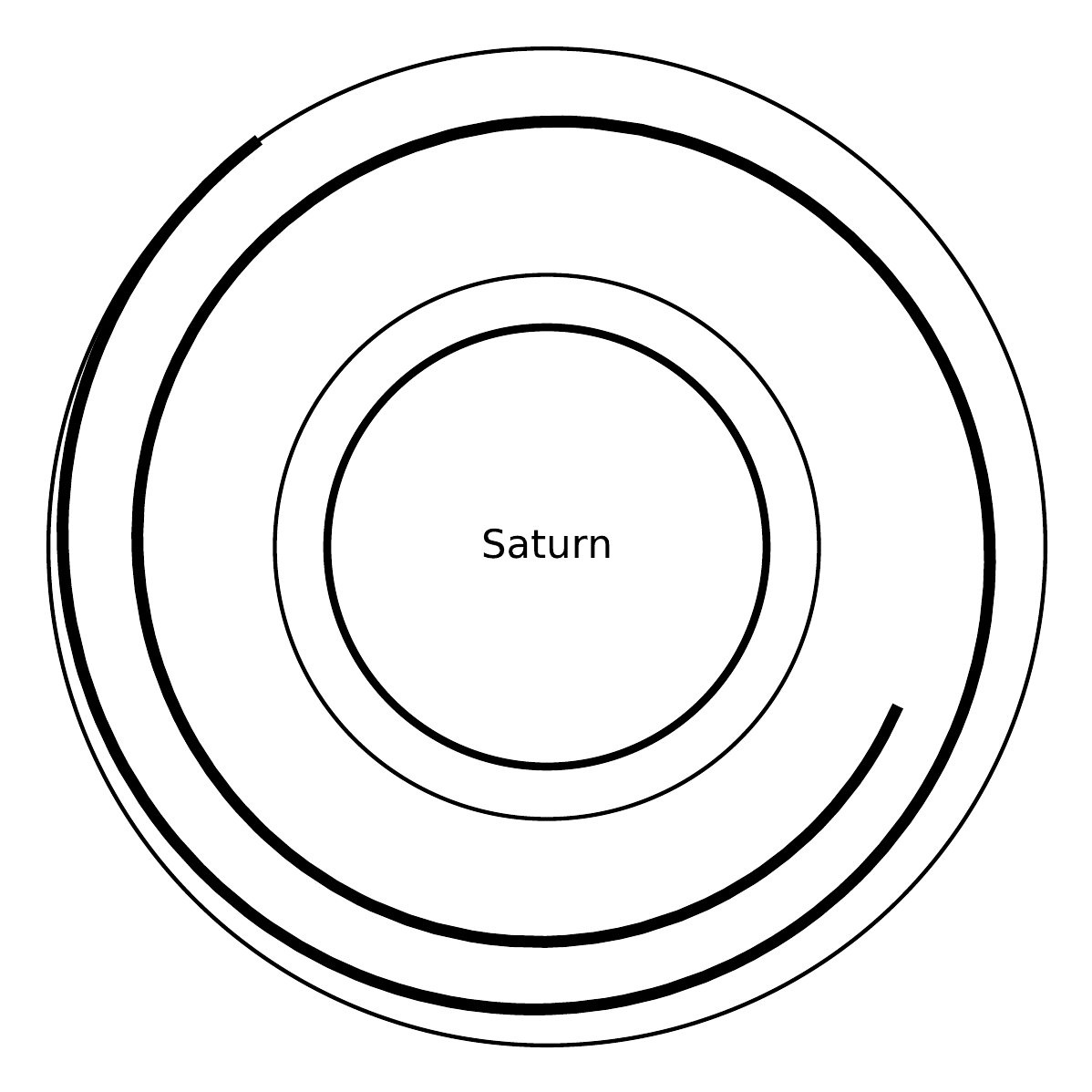}
\caption{Equirepartition of a cloud of dust in Saturn's rings. On the left: the cloud (thick black line) at initial time. 
In the middle: the same cloud, after 6 hours. On the right: the same cloud, after 48 hours.}
\label{fig:Saturne}
\end{figure}

Keplerian shear is a more general feature of many integrable Hamiltonian dynamical systems. 
Using action-angle coordinates, the phase space is foliated by invariant Lagrangian tori, 
and the dynamics of a point belonging to the phase space is conjugate to a translation on one of 
these tori. Provided that the translations on the Lagrangian tori are (in some sense) asynchronous, 
the dynamics shear the transversals to the invariant tori, so that in large time, 
densities equidistribute along the tori. In the case of planetary rings, the invariant tori 
are orbits of given radius, and the asynchronicity comes from the variation of the orbital period: 
we recover classical Keplerian shear. Other systems with Keplerian shear are the geodesic 
flow on a flat torus (see Fig~\ref{fig:Tore}), or the dynamics of a ball bouncing in a square box.

\begin{figure}[h!]
\vspace{1cm}
\centering
\includegraphics[scale=0.4]{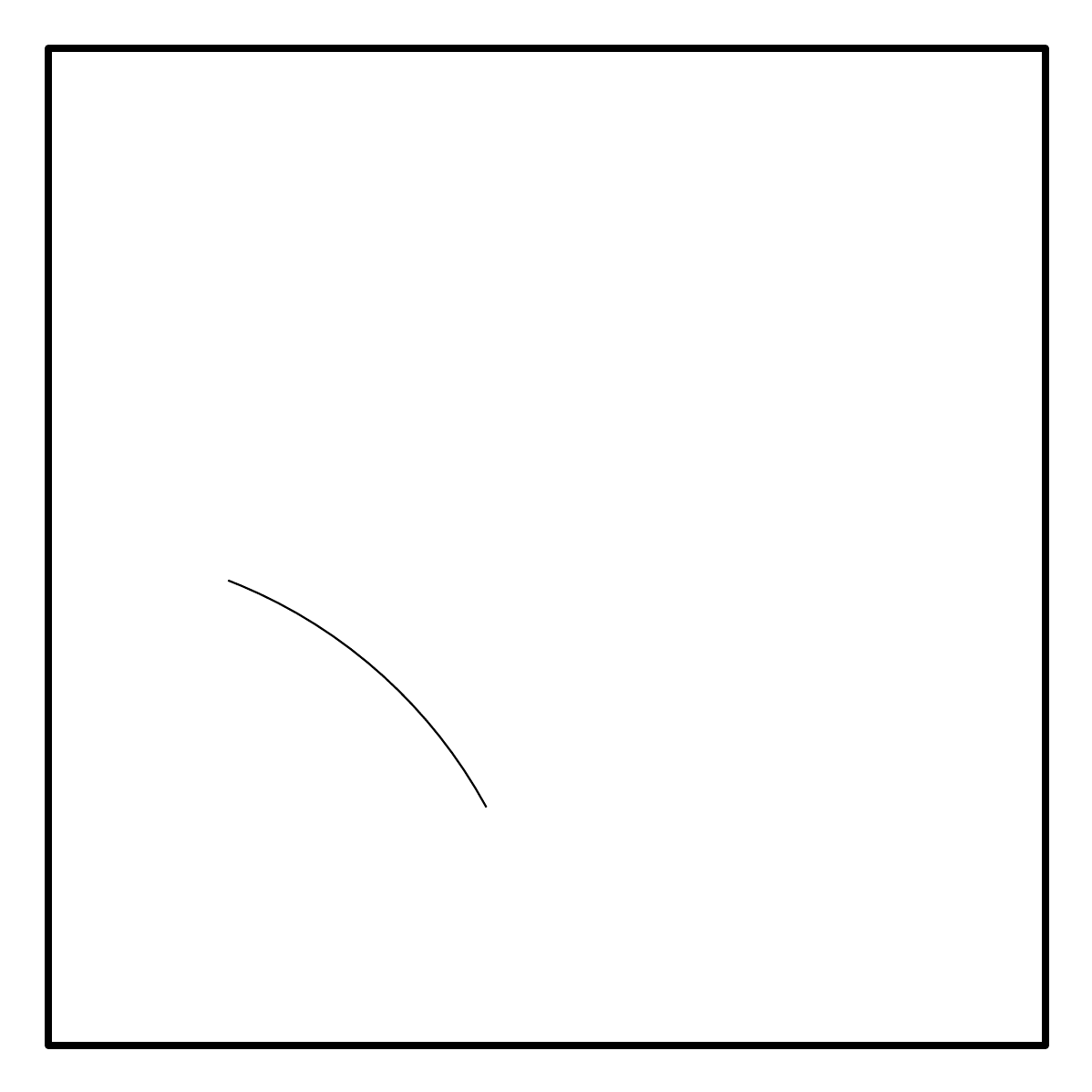}
\includegraphics[scale=0.4]{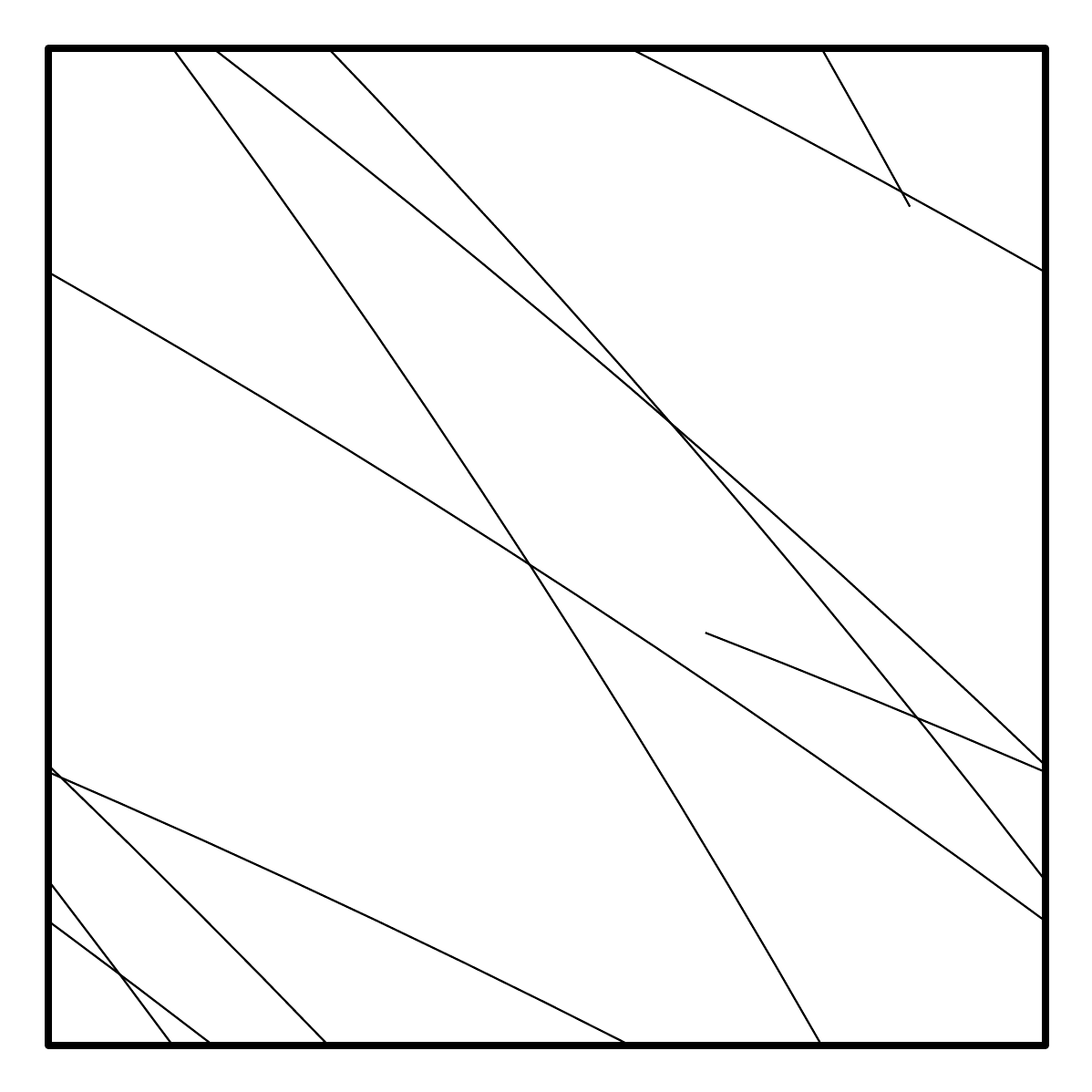}
\includegraphics[scale=0.4]{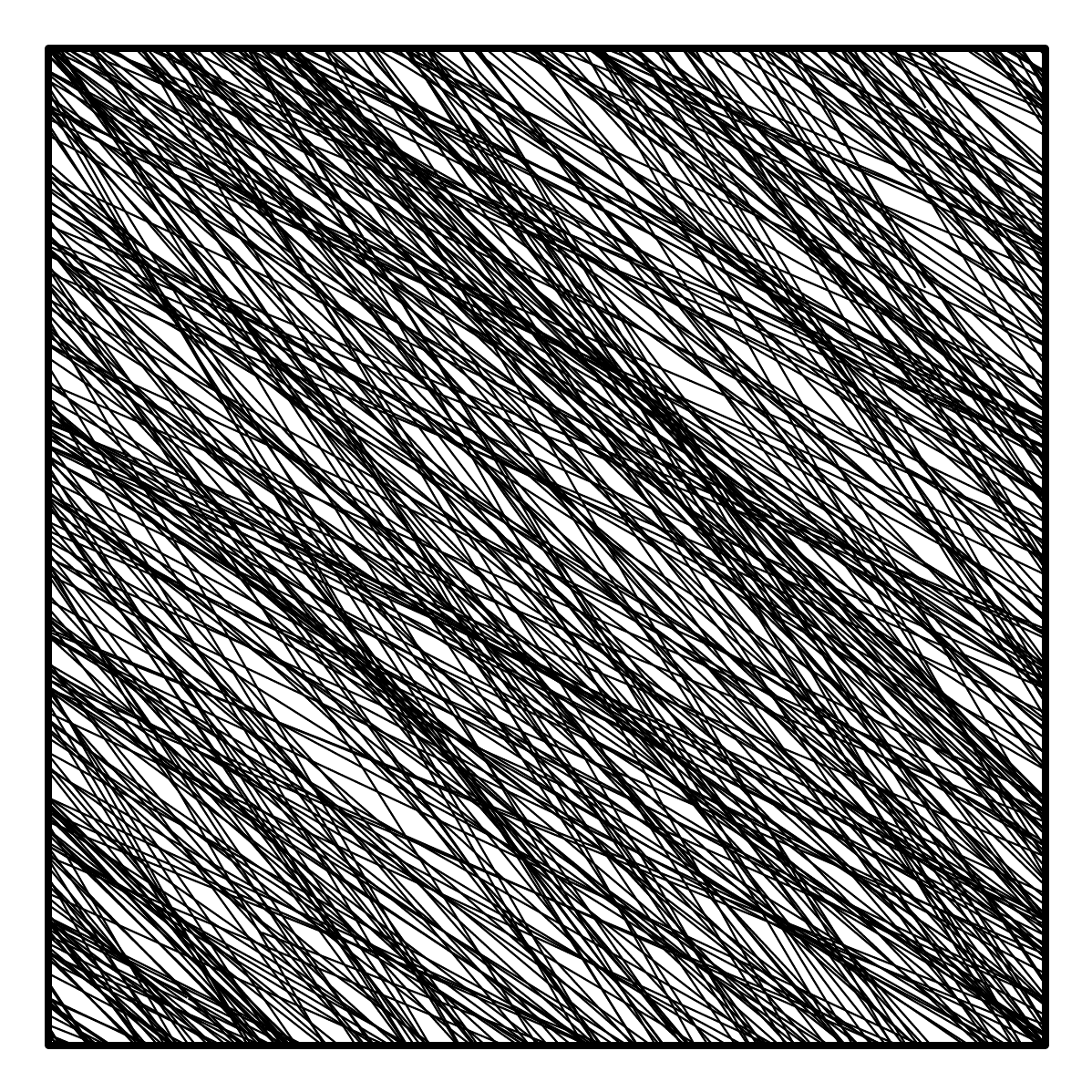}
\caption{Propagation of a wavefront at unit speed in a unit square torus. The wave starts from the corner, and propagates at unit speed. 
On the left: the wavefront at time $0.5$. In the middle: the wavefront at time $10$. On the right: the wavefront at time $500$.}
\label{fig:Tore}
\end{figure}

\smallskip

In this article, we frame Keplerian shear in the more general context of ergodic theory, 
as a conditional version of the notion of strong mixing.

\begin{definition}[Keplerian shear]\quad

A dynamical system $(\Omega, \mu, (g_t)_{t \in \R})$ which preserves a probability measure is said to exhibit 
\textit{Keplerian shear} if, for all $f \in \Lbb^2 (\Omega, \mu)$,
\begin{equation}
\label{eq:DefinitionKeplerianShear}
\lim_{t \to + \infty} f \circ g_t 
= \Ebb_\mu (f | \Ical),
\end{equation}
where $\Ical$ is the invariant $\sigma$-algebra and the convergence is for the weak topology on $\Lbb^2 (\Omega, \mu)$.
\end{definition}

Recall that a system $(\Omega, \mu, (g_t)_{t \in \R})$ is mixing if and only if, for any function 
$f \in \Lbb^2 (\Omega, \mu)$,
\begin{equation*}
\lim_{t \to + \infty} f \circ g_t 
= \int_\Omega f \dd \mu 
= \Ebb_\mu (f),
\end{equation*}
where the limit is taken in the weak topology on $\Lbb^2 (\Omega, \mu)$, so a system 
$(\Omega, \mu, (g_t)_{t \in \R})$ is mixing if and only if it is ergodic 
and exhibits Keplerian shear. As such, Keplerian shear is a conditional version of the notion 
of strong mixing. Informally, if the system restricted to its invariant subsets is mixing, 
then $(\Omega, \mu, (g_t)_{t \in \R})$ has Keplerian shear. The interesting examples occur 
when these restrictions are ergodic, but not mixing: that is the case, 
for instance, of translation flows on a torus.

\smallskip

In this article, we give a criterion ensuring Keplerian 
shear for a large class of such systems; for instance, one of our result is:

\begin{proposition}[Corollary of Theorem~\ref{thm:ClassiquePrincipal} and Proposition~\ref{prop:BaireGenericite}]

Let $M$ be a Riemannian manifold, $d \geq 1$ and $k \in [1, \infty]$. Let $v \in \Ccal^k (M, \R^d)$, 
and put $g_t (x,y) := (x, y+tv(x))$ for $(x,y)\in M \times \Tbb^d$. If:
\begin{equation*}
\Vol_M (\nabla \langle \xi, v \rangle) = 0 \quad \forall \xi \in \Z^d \setminus \{0\},
\end{equation*}
then $(M \times \Tbb^d, \Vol_M \otimes \Leb_{\Tbb^d}, (g_t))$ exhibits Keplerian shear. 
Moreover, the criterion above is satisfied for a generic $v \in \Ccal^k (M, \R^d)$.
\end{proposition}

We also study the rate of decay of conditional covariance for the geodesic flow on $T^1 \Tbb^d$, 
and give non-trivial examples of non-Hamiltonian systems with Keplerian shear.

\smallskip

Keplerian shear for the geodesic flow on the flat torus is related to two famous 
problems. The first is Landau's damping for plasma dynamics on a torus (see Landau's 
article~\cite{Landau:1946}, and~\cite[Theorem~3.1]{MouhotVillani:2011} for a version 
which follows closely our formalism), where the effect is qualitatively similar,
although the underlying mechanism is different. The second is Gauss's circle problem, 
which consists in counting integral points in a large disc; we shall discuss 
it in Sub-subsection~\ref{subsubsec:Gauss}. The methods used to tackle these problems 
are either through Fourier transform (e.g. for Landau damping), or with a 
big arc/small arc decomposition (typical for Gauss's circle problem). While both work in our setting, 
we shall only use the Fourier transform.

\smallskip

In the context of ergodic theory, a notion closely related with Keplerian shear was used independently 
by F.~Maucourant~\cite{Maucourant:2016} to prove that the some hyperbolic actions on 
$(\R^d \rtimes SL_d (\R))_{/(\Z^d \rtimes SL_d (\Z))}$ are ergodic for a large class of measures. 
The presentation in~\cite{Maucourant:2016} is however very different, as the phenomenon 
-- named \textit{asynchronicity} -- is described as a version of unique ergodicity for 
measures with prescribed marginals.

\medskip
\textbf{\textsc{Organization of the article}}
\smallskip

Section~\ref{sec:ProprietesFondamentales} gives general 
results on the notion of Keplerian shear (including equivalences between distinct definitions), 
and gives us some tools to use for the remainder of the article.

\smallskip

Section~\ref{sec:FibreBundle} 
deals with a first family of systems which may exhibit Keplerian shear: fibrations by tori, 
where the flow acts by translation on each torus. using action-angle coordinates, this family includes 
integrable Hamiltonian flows. We give an explicit criterion ensuring Keplerian shear, check that 
it is $\Ccal^r$-generic ($r \geq 1$) and satisfied for some explicit systems, then give rates of convergence 
for the geodesic flow on $T^1 \Tbb^n$. We also detail the link between Keplerian shear and 
the unique ergodicity as investigated in~\cite{Maucourant:2016}.

\smallskip

Section~\ref{sec:StretchedBirkhoff} deals with another family of dynamical systems (roughly, 
``fibrations by suspension flows''), which includes many non-Hamiltonian examples, and uses 
a different mechanism to ensure Keplerian shear.

\smallskip

The shorter Section~\ref{sec:SansMelange} 
gives examples of systems without Keplerian shear.

\medskip
\textbf{\textsc{A note on the terminology}}
\smallskip

Given that Keplerian shear is a conditional version of the notion mixing, one could want 
to use a terminology such as \textit{conditional (strong) mixing}. We prefer to eschew this option, 
and to keep the name of Keplerian shear; indeed, we think that otherwise the name of conditional 
(strong) mixing would be overloaded.

\smallskip

Indeed, in probability theory, there are already multiple notions of conditional mixing; 
compare for instance~\cite{Rao:2009} (where it refers to conditional $\alpha$-mixing) 
and~\cite{KacemLoiselMaumeDeschamps:2016}, among others. 

\smallskip

More worryingly, in ergodic theory, the notion of conditionally weakly mixing systems 
is well-established (see e.g.~\cite{Tao:2008}), but if one where to conceive a notion of 
conditional strong mixing along this line, the resulting notion would be stronger than Keplerian shear, 
essentially requiring that almost every subsystem in its ergodic decomposition be mixing.

\medskip
\textbf{\textsc{Open problems}}
\smallskip

We sum up here some further leads which seem worth pursuing.

\smallskip

The setting of Section~\ref{sec:FibreBundle} covers integrable Hamiltonian systems. However, 
it requires some regularity, and in particular it does not cover singular systems. A conjecture by 
Boshernitzan asserts that given a compact translation surface $S$, the geodesic flow on 
$(T^1 S, \Liouv)$ exhibits Keplerian shear. This question, mentioned as \textit{illumination by circles}, 
also appears in~\cite{Monteil:2005}, and admits a partial answer by J.~Chaika and P.~Hubert~\cite{ChaikaHubert:2015}, 
where the convergence of $\Cov (f, g \circ g_t|\Ical)$ to zero is shown along a density $1$ 
subsequence for all continuous observables $f$ and $g$\footnote{Technically, J.~Chaika and P.~Hubert 
show the convergence only for observables which do not depend on the direction, but 
a straightworward generalization and a diagonal argument yield the general case.}.

\smallskip

In Subsection~\ref{subsec:Speed}, we investigate the speed of Keplerian shear for the geodesic flow 
on $T^1 \Tbb^n$. The problem is simplified by the particularities of the geometry of the sphere, 
more precisely the fact that its principal curvatures do not vanish. What would the speed of convergence 
be if the curvature vanishes (e.g.\ in a topologically or measure-theoretically generic setting)?

\smallskip

Finally, while the settings of Sections~\ref{sec:FibreBundle} and~\ref{sec:StretchedBirkhoff} 
are distinct, it could be that they are a special case of a more general structure. A natural candidate 
would be spaces fibrated by suspension tori, but we need new tools to prove Keplerian shear (or even to 
get a description of the invariant $\sigma$-algebra $\Ical$).

\medskip
\textbf{\textsc{Acknowledgements}}
\smallskip

I would like to thank S\'ebastien Gou\"ezel, Bassam Fayad and Fra\c{c}ois Maucourant for 
their useful comments and some of the references, as well as J\'er\^ome Buzzi for his 
feedback on the presentation.

\section{General properties of Keplerian shear}
\label{sec:ProprietesFondamentales}

The following lemma from basic functional analysis is quite useful to prove the ergodicity and mixing of any given dynamical system, 
and will be instrumental in the remainder of our article.

\begin{lemma}\quad
\label{lem:LemmeAnalyseFonctionnelle}

Let $\B$ be a Banach space. Let $(T_t)_{t \geq 0}$ be a family of operators on $\B$, such that $\sup_{t \in \R_+} \norm{T_t}{\B \to \B} < + \infty$. 
Let $T$ be an operator on $\B$.

Let $E$ and $E^*$ be subsets of $\B$ and $\B^*$ respectively, whose span is dense in their respective space. Assume that, for all $f \in E$ and 
$g \in E^*$, 
\begin{equation}
\label{eq:ConvergenceFaible}
\lim_{t \to + \infty} \langle g, T_t f \rangle 
= \langle g, T f \rangle.
\end{equation}

Then $(T_t f)_{t \geq 0}$ converges weakly to $T f$ for all $f \in \B$.
\end{lemma}

\begin{proof}\quad

By bilinearity, Equation~\eqref{eq:ConvergenceFaible} holds for all $f \in \VectSpan(E)$ and $g \in \VectSpan(E^*)$.

\smallskip

Since $\sup_{t \in \R_+} \norm{T_t}{\B \to \B} < + \infty$, the family of functions $T_t : \B^* \times \B \to \C$ 
is locally equicontinuous, and by the remark above, it converges to $T$ on a dense subset. Hence, the convergence 
of Equation~\eqref{eq:ConvergenceFaible} holds for all $f \in \B$ and $g \in \B^*$.
\end{proof}


When we use Lemma~\ref{lem:LemmeAnalyseFonctionnelle}, the operator $T_t$ shall correspond to the 
composition by the flow $g_t$ at time $t$, and the operator $T$ to the projection $f \mapsto \Ebb (f | \Ical)$. 
Since the flow is assumed to preserve the measure, for all $t \geq 0$ and all $p \in [1, + \infty]$, the operator 
$T_t$ acting on $\Lbb^p (\Omega, \mu)$ is unitary. Lemma~\ref{lem:LemmeAnalyseFonctionnelle} implies that 
to prove the Keplerian shear in one of those Banach space $\B$ (potentially different from $\Lbb^2$), 
it is enough to restrict ourselves to subsets $E$ of $\B$ and $E^*$ of  $\B^*$ whose linear span is dense.
As a first consequence, in the definition of Keplerian shear, one may replace $\Lbb^2$ by $\Lbb^p$ for any 
$p \in [1, +\infty)$:

\begin{proposition}\quad

Let $(\Omega, \mu, (g_t)_{t \in \R})$ be a flow which preserves a probability measure. Let $\Ical$ be the invariant 
$\sigma$-algebra of the system. Then there is equivalence between:
\begin{itemize}
\item There exists $p \in [1, + \infty)$ such that, for all $f \in \Lbb^p (\Omega, \mu)$, we have $f \circ g_t \to \Ebb (f | \Ical)$ weakly in $\Lbb^p$.
\item The system exhibits Keplerian shear.
\item For all $p \in [1, + \infty)$, for all $f \in \Lbb^p (\Omega, \mu)$, we have $f \circ g_t \to \Ebb (f | \Ical)$ weakly in $\Lbb^p$.
\end{itemize}
\end{proposition}

\begin{proof}

We only prove the non-trivial implication. Let $p \in [1, + \infty)$. Assume such that, for all $f \in \Lbb^p (\Omega, \mu)$, 
we have $f \circ g_t \to \Ebb (f | \Ical)$ weakly in $\Lbb^p$. Then, since $\Lbb^\infty \subset \Lbb^p \cap (\Lbb^p)^*$, 
for all $f_1$ and $f_2$ in $\Lbb^\infty$,
\begin{equation*}
\lim_{t \to + \infty} \langle f_1,  f_2 \circ g_t \rangle 
= \langle f_1, \Ebb (f_2 | \Ical) \rangle.
\end{equation*}
Let $q \in [1, + \infty)$. Since $\Lbb^\infty$ is dense in both $\Lbb^q$ and $(\Lbb^q)^*$, by Lemma~\ref{lem:LemmeAnalyseFonctionnelle}, 
the convergence above occurs for all $f_1$ and $f_2$ in $\Lbb^q$ and $(\Lbb^q)^*$ respectively.
\end{proof}

A second consequence is that Keplerian shear is not uniquely a property of the invariant measure $\mu$, but 
of the class of $\mu$.

\begin{proposition}\quad
\label{prop:MesuresAbsContinues}

Let $(\Omega, \mu, (g_t)_{t \in \R})$ be a flow which preserves a probability measure and exhibits Keplerian shear. 
Let $\nu \ll \mu$ be a probability measure which is also $(g_t)$-invariant. Then $(\Omega, \nu, (g_t)_{t \in \R})$ also 
exhibits Keplerian shear.
\end{proposition}

\begin{proof}

Let $(\Omega, \mu, (g_t)_{t \in \R})$ and $\nu$ be as in assumptions of the proposition. Let $h := \de \nu / \de \mu$.
Let $f_1$ be in $\Lbb^\infty (\Omega, \mu)$ be such that $f_1h \in \Lbb^2 (\Omega, \mu)$, and let $f_2 \in \Lbb^\infty (\Omega, \mu)$.
Since $h$ is $\Ical$-measurable, $\nu$-almost surely, $\Ebb_\nu (f_2 | \Ical) = \Ebb_\mu (f_2 | \Ical)$. Let $t \geq 0$. Then:
\begin{equation*}
\Ebb_\nu (f_1 \cdot f_2 \circ g_t) 
= \Ebb_\mu ((f_1h) \cdot f_2\circ g_t) .
\end{equation*}
Since the initial system is assumed to have Keplerian shear, $f \in \Lbb^\infty (\Omega, \mu)$ and $gh \in \Lbb^\infty (\Omega, \mu)$, 
we get:
\begin{equation*}
\lim_{t \to + \infty} \Ebb_\nu (f_1 \cdot f_2 \circ g_t) 
= \Ebb_\mu (f_1h\Ebb_\mu (f_2 | \Ical)) 
= \Ebb_\nu (f_1 \Ebb_\nu (f_2 | \Ical)).
\end{equation*}
The canonical projection $\Lbb^\infty (\Omega, \mu) \to \Lbb^\infty (\Omega, \nu)$ is surjective, so its image is dense in $\Lbb^2 (\Omega, \nu)$. 
The image of the set of functions $f_1 \in \Lbb^\infty (\Omega, \mu)$ such that $f_1h \in \Lbb^2 (\Omega, \mu)$ by this projection is 
also dense in $\Lbb^2 (\Omega, \nu)$. We use Lemma~\ref{lem:LemmeAnalyseFonctionnelle} to conclude.
\end{proof}

The last lemma asserts that, in the definition of Keplerian shear, the limit object $\Ebb_\mu (f | \Ical)$ cannot be meaningfully modified.

\begin{proposition}\quad

Let $(\Omega, \mu, (g_t)_{t \in \R})$ be a flow which preserves a probability measure. Let $f, h \in \Lbb^2 (\Omega, \mu)$.

\smallskip

If $(f \circ g_t)_{t \in \R}$ converges weakly to $h$, then $h = \Ebb_\mu (f | \Ical)$.
\end{proposition}

\begin{proof}

Let $g \in \Lbb^2 (\Omega, \mu)$. Our hypotheses imply that $\lim_{t \to + \infty} \Ebb_\mu (g \cdot f \circ g_t) = \Ebb_\mu (gh)$. 
In addition, the function $t \to \Ebb_\mu (g \cdot f \circ g_t)$ is measurable and bounded. By taking the Ces\`aro average, we get:
\begin{equation*}
\lim_{t \to + \infty} \Ebb_\mu \left( \frac{g}{t} \int_0^t f \circ g_s \dd s \right) 
= \lim_{t \to + \infty} \frac{1}{t} \int_0^t \Ebb_\mu (g \cdot f \circ g_s) \dd s 
= \Ebb_\mu (gh).
\end{equation*}
On the other hand, by von Neumann's ergodic theorem, 
\begin{equation*}
\lim_{t \to + \infty} \Ebb_\mu \left( \frac{g}{t} \int_0^t f \circ g_s \dd s \right) 
= \Ebb_\mu (g \Ebb_\mu (f | \Ical)).
\end{equation*}
Since this holds for all $g \in \Lbb^2$, we have $h = \Ebb_\mu (f | \Ical)$.
\end{proof}

%
%
%
%

\section{Affine tori bundles}
\label{sec:FibreBundle}

\subsection{Setting and main theorem}
\label{subsec:TheoremePrincipal}

We generalize our introductory examples to a class of flows on fibre bundles by tori which leave the basis invariant. 
More specifically, the spaces on which we work are the following:

\begin{definition}

An affine tori bundle is a $\Ccal^1$ manifold $\Omega$ which is a fiber 
bundle by $d$-dimensional tori, with group structure $\Tbb^d \rtimes \GL_d (\Z)$. 
In other words, there exist:
\begin{itemize}
\item two integers $n$, $d \geq 1$;
\item a $n$-dimensional $\Ccal^1$ real manifold $M$;
\item a $\Ccal^1$ projection $\pi : \ \Omega \to M$;
\item a maximal atlas $\Acal$ on $M$,
\end{itemize}
such that, for all $U \in \Acal$, we have a diffeomorphism $\psi_U : \pi^{-1} (U) \to U \times \Tbb^d$ 
such that $\pi_1 \circ \psi_U = \pi$, and the change of charts are given by:
\begin{equation*}
\psi_V \circ \psi_U^{-1} :  \left\{ 
\begin{array}{lll}
(U \cap V) \times \Tbb^d & \to & (U \cap V) \times \Tbb^d \\
(x,y) & \mapsto & (x, \alpha_{U, V} (x) + A_{U, V} (y))
\end{array}
\right. ,
\end{equation*}
where $\alpha_{U,V}$ is $\Ccal^1$ and $A_{U, V} \in \GL_d (\Z)$.
\end{definition}

The notions of ``subset of zero Lebesgue measure'' or ``subset of full Lebesgue measure'' 
are well-defined on $\Ccal^1$ manifolds (as they are invariant by diffeomorphisms), 
and thus so is the notion of ``probability measure absolutely continuous with respect to the 
Lebesgue measure''. We will abuse notations and write $\Leb (A) = 0$ for a measurable subset of zero 
Lebesgue measure $A$, and $\mu \ll \Leb$ for an absolutely continuous measure.

\begin{definition}

Let $\Omega$ be an affine tori bundle. 
A flow $(g_t)_{t \in \R}$ on $\Omega$ is said to be \textit{compatible on a chart} 
$\psi_U : \ \pi^{-1} (U) \to U \times \Tbb^d$ if 
there exists $v_{\psi_U} \in \Ccal^1 (U, \R^d)$ such that, for all $t \in \R$,
\begin{equation*}
\psi_U \circ g_t \circ \psi_U^{-1} (x,y) 
= (x, y + t v_\psi (x)).
\end{equation*}

A $\sigma$-finite measure $\mu$ on $\Omega$ is said to be \textit{compatible on a chart} $\psi_U : \ \pi^{-1} (U) \to U \times \Tbb^d$ 
if $\psi_{U,*} \mu_{|\pi^{-1} (U)} = (\pi_* \mu)_{|U} \otimes \Leb_{\Tbb^d}$.

\smallskip

A flow or a measure is said to be \textit{compatible} if it is compatible on all charts.
\end{definition}

A compatible measure is always invariant under a compatible flow. In addition, 
this notion behaves well with respect to the affine structure on the manifolds we work with. 
If a flow or a measure is compatible on some chart $\psi_U: U \cap V \to \pi (U \cap V) \times \Tbb^d$ 
and if $\psi_{U,V}$ is a change of charts, then the flow or the measure is compatible on the 
chart $\psi_{V|U \cap V}: U \cap V \to \pi (U \cap V) \times \Tbb^d$.

\smallskip

In what follows, we are working mostly with absolutely continuous measures. In this case, what happens 
on a subset of zero Lebesgue measure does not matter: the assumption that $M$ be a manifold can be weakened 
to account for singularities or boundaries. 

\smallskip

In light of the previous paragraph, the introduction of the structure group $\Tbb^d \rtimes \GL_d (\Z)$ 
might look gratuitous: one can always cut out the manifold $M$ along a set of zero Lebesgue measure to get a disjoint 
union of simply connected domains, on which there is no holonomy. However, this structure appears 
naturally in many examples. For instance, for all $n \geq 1$, we can work with the geodesic flow on 
$T \Sbb_n$: if we ignore the set of null tangent vectors, which is negligible, we get a fibre bundle 
over $\R_+^* \times \widetilde{Gr} (2, n+1)$ with fibre $\Sbb_1$. With the same adaptation, our setting 
also includes billiards in ellipsoids or the geodesic flow on ellipsoids (see C.~Jacobi~\cite{Jacobi:1866} 
for the geodesic flow on ellipsoids, J.~Moser~\cite{Moser:1978} for similar examples, and S.~Tabachnikov~\cite{Tabachnikov:2002} 
for the relation between the geodesic flow and the billiard). Let us also mention the study of the geodesic flow 
on $(\R^d \rtimes \SL_d (\R))_{/ \Z^d \rtimes\SL_d (\Z)}$ done by F.~Maucourant~\cite{Maucourant:2016}, 
in which the same structure appears.

\smallskip

Another important remark is that, when we change charts from chart $U$ to chart $V$, we have 
$v_{\psi_V | U \cap V} = A_{U, V} v_{\psi_U | U \cap V}$. So, while there is in general 
no well-defined function $v : \ M \to \R^d$ which gives the direction of the flow, the set of functions 
$\{x \mapsto \langle \xi, v(x) \rangle\}_{\xi \in \Z^d - \{0\}}$ is well-defined.

\smallskip

We are now ready to state our main theorem.

\begin{theorem}\quad
\label{thm:ClassiquePrincipal}

Let $\pi: \Omega \to M$ be an affine $d$-dimensional tori bundle over a manifold $M$. Let $(g_t)_{t \in \R}$ 
be a compatible flow, and $\mu$ be an absolutely continuous compatible probability measure.

\smallskip

If $\Leb (\bigcup_{\xi \in \Z^d - \{0\}} \{ \de \langle \xi, v \rangle = 0\}) = 0$ on $M$, then the dynamical system 
$(\Omega, \mu, (g_t))$ exhibits Keplerian shear.
\end{theorem}

\begin{proof}\quad

Assume that $\Leb (\bigcup_{\xi \in \Z^d - \{0\}} \{ \de \langle \xi, v \rangle = 0\}) = 0$. 
Then $\Leb (\bigcup_{\xi \in \Z^d - \{0\}} \{ \langle \xi, v \rangle = 0\}) = 0$, so $(g_t (x,y))_{t \in \R}$ 
equidistributes in $\{x\} \times \Tbb^d$ for Lebesgue-almost every $x$. Hence, up to completion 
by the measure $\mu$, the invariant $\sigma$-algebra of the flow is $\Ical := \pi^* \Bcal_M$, 
where $\Bcal_M$ is the Borel $\sigma$-algebra of $M$.

\smallskip

Our goal is to find a family of observables which is large enough to generate a dense subset of $\Lbb^2 (\Omega, \mu)$, 
and specific enough to make our computations manageable. Roughly, we choose a specific frequency in 
the direction of the torus $\Tbb^d$. Under the hypothesis of the theorem, we can rectify the 
differential form $\langle \xi, v \rangle$ so that it has a very simple expression. Then we  
choose observables which split into an observable $a$ in the direction of $\langle \xi, v \rangle$, 
and another observable $b$ in the direction of the kernel. The later observable $b$ does not see the shearing at all, 
so the shearing only affects $a$.

\smallskip

Let $(U_i, \varphi_i)_{i \in I}$ be a countable cover of $M$ by disjoint open 
charts\footnote{The goal of this first decomposition is only to get well-defined speed 
functions $v_{\psi_i}$, and can be bypassed if the fibre bundle is trivial.}, 
up to a Lebesgue-negligible set, with $\varphi_i : U_i \to W_i \subset \R^n$. 
Let $\psi_i : \pi^{-1} (U_i) \to U_i \times \Tbb^d$ be a family of triviliazing charts 
for $\Omega$, and let $v_i := v_{\psi_i}$.

\smallskip

For $\xi \in \Z^d \setminus \{0\}$, let $V_i^\xi := V_i \cap \{ \de \langle \xi, v_i \rangle = 0\}$. 
Using the local normal form of submersions, we can find a finite or countable family $(V_{ij}^\xi)_{j \in J(i,\xi)}$ 
of open sets which are pairwise disjoint, cover $V_i^\xi$ up to a Lebesgue-negligible set, and with charts 
$\varphi_{ij}^\xi : V_{ij}^\xi \to W_{ij}^\xi \subset \R^n$ such that $\langle \xi, v_i \rangle \circ \varphi_{ij}^{\xi,-1} (x) = x_1$.
For $\xi=0$, we choose $J(i,\xi)$ to be a singleton and take $V_{ij}^0:=V_i$.

\smallskip

Given a point $p \in \R^n$, we write $p_x$ its first coordinate in $\R^n$, 
and $p_y$ for its remaining $n-1$ coordinates in $\R^n$. Given a point $p \in M \times \Tbb^d$, 
we write $p_z$ for its coordinate in $\Tbb^d$. We apply Lemma~\ref{lem:LemmeAnalyseFonctionnelle}, 
with the Banach space $\Bcal = \Bcal^* = \Lbb^2 (\Omega, \mu)$, and:
\begin{equation*}
E 
= E^* 
= \bigcup_{\substack{i \in I \\ \xi \in \Z^d \\ j \in J (i,\xi)}} \left\{a ((\varphi_{ij}^\xi \circ \pi)_x) b ((\varphi_{ij}^\xi \circ \pi)_y) e^{2\pi i \langle \xi, \psi_{i,z} \rangle} : \ a, b \in \Lbb^\infty, ab \in \Lbb^\infty (W_{ij}^\xi, \Leb) \right\}.
\end{equation*}

\smallskip

Let $f_j = a_j b_j e^{\langle \xi_j, \cdot \rangle}$, with $j \in \{1,2\}$, be in $E$.
If the corresponding indices $i \in I$ are different, then $f_1$ and $f_2 \circ g_t$ 
have disjoint support for all $t$, so $\Ebb_\mu (\overline{f}_1 \cdot f_2 \circ g_t) = 0 = \Ebb_\mu (\overline{f}_1 \cdot \Ebb_\mu (f_2 | \Ical))$ 
for all $t \in \R$. We can thus assume without loss of generality that they are supported 
by the same open set $\pi^{-1} (V_i)$.

\smallskip

If the corresponding frequencies $\xi_j \in 2\pi\Z^d$ are different, then the integral of $\overline{f}_1 \cdot f_2 \circ g_t$ 
on each torus $\Tbb^d$ vanishes, and a least one of $\Ebb_\mu (\overline{f}_2 | \Ical)$ or $\Ebb_\mu (f_2 | \Ical)$ 
vanishes, so for all $t \in \R$:
\begin{equation*}
\Ebb_\mu (\overline{f}_1 \cdot f_2 \circ g_t) 
= 0 
= \Ebb_\mu (\Ebb_\mu (\overline{f}_1 | \Ical) \Ebb_\mu (f_2 | \Ical)) 
= \Ebb_\mu (\overline{f}_1 \cdot \Ebb_\mu (f_2 | \Ical)).
\end{equation*}
We can thus assume without loss of generality that their frequencies $\xi_j$ are the same; let us denote it by $\xi$. 
If $\xi = 0$, then $f_1$ and $f_2$ are invariant under the flow, so there is nothing more to prove. 
We further assume that $\xi \neq 0$.

\smallskip

If the corresponding indices $j \in J(i, \xi)$ are different, then the supports of $f_1$ and $f_2 \circ g_t$ 
are disjoint for all $t$, so then again there is nothing more to prove. 
We thus fruther assume that these indices are the same.

\smallskip

Write $h_{ij}^\xi := \de (\varphi_{ij,*}^\xi \pi_* \mu) / \de \Leb \in \Lbb^1 (W_{ij}^\xi, \Leb)$. Then, for all $t \in \R$:
\begin{align*}
\Ebb_\mu (\overline{f}_1 \cdot f_2 \circ g_t) & = \int_{W_{ij}^\xi} \overline{a}_1(x) \overline{b}_1(y) a_2(x) b_2(y) \int_{\Tbb^d} e^{2\pi i \langle \xi, z+tv_i \circ \varphi_{ij}^{-1}(x,y)-z \rangle} \dd z h_{ij}^\xi (x,y) \dd x \dd y \\
& = \int_{W_{ij}^\xi \cap \{0\} \times \R^{n-1}} (\overline{b}_1 b_2) (y) e^{2\pi i \langle \xi, tv_i \circ \varphi_{ij}^{-1}(0,y) \rangle} \int_{W_{ij}^\xi \cap (y+\R\times\{0\})} (\overline{a}_1 a_2)(x) e^{ixt} h_{ij}^\xi (x,y) \dd x \dd y.
\end{align*}
The function $x \mapsto (\overline{a}_1 a_2)(x) h_{ij}^\xi (x,y)$ is integrable for almost every $y$. 
By the Riemann-Lebesgue lemma, the inner integral decay to $0$ as $t \to \pm \infty$. The inner integral 
is bounded by:
\begin{equation*}
\norm{\overline{a}_1 a_2}{\Lbb^\infty} \int_{W_{ij}^\xi \cap (y+\R\times\{0\})} h_{ij}^\xi (x,y) \dd x,
\end{equation*}
which is integrable as a function of $y$. Hence, by the dominated convergence theorem, 
\begin{equation*}
\lim_{t \to \pm \infty} \Ebb_\mu (\overline{f}_1 \cdot f_2 \circ g_t) 
= 0 
= \Ebb_\mu (\overline{f}_1 \cdot \Ebb_\mu (f_2| \Ical)). \qedhere
\end{equation*}
\end{proof}

\subsection{Genericity}
\label{subsec:FlotTM}

We check in this subsection that the sufficient condition in Theorem~\ref{thm:ClassiquePrincipal} 
is $\Ccal^r$-generic for all $r \in [1, +\infty]$. Given a $\Ccal^r$ affine 
tori bundle $\Omega$, we begin by endowing the space of $\Ccal^r$ compatible flows with a topology.

\smallskip

Let $r \in [1, +\infty]$, and $\pi: \Omega \to M$ be a $\Ccal^r$ affine $d$-dimensional tori bundle 
over a manifold $M$. Let $(U_i)_{i \in I}$ be a locally finite open cover of $M$ with trivializing charts 
$\varphi_i : U_i \to W_i \subset \R^n$. Let $(K_i)_{i \in I}$ be a cover of $M$ by compact sets 
subordinated to $(U_i)_{i \in I}$.

\smallskip

Denote by $\Fcal^r (M, \R^d)$ the set of $\Ccal^r$ compatible flows on $\Omega$. For each $v \in \Fcal^r (M, \R^d)$, 
there is a unique family of function $(v_i)_{i \in I}$ which generates the flow, where each $v_i$ 
belongs to $\Ccal^r (U_i,\R^d)$. A sequence $(v_n)$ of elements of $\Fcal^r (M, \R^d)$ converges 
to $v \in \Fcal^r (M, \R^d)$ if, for all $i \in I$, all the derivatives of $(v_{n,i})_{n \geq 0}$ (up to order $r$) 
converge to those of $v$ uniformly on each $K_i$. This topology does not depends on the choice of 
the charts $(U_i)_{i \in I}$ nor on that of the compacts $(K_i)_{i \in I}$, and makes $\Fcal^r (M, \R^d)$ 
a Baire space.

\begin{proposition}\quad
\label{prop:BaireGenericite}

Let $r \in [1, +\infty]$. Let $\pi: \Omega \to M$ be a $\Ccal^r$ affine $d$-dimensional tori bundle over a manifold $M$. 

\smallskip

For a Baire generic subset of compatibles flows in $\Fcal^r (M, \R^d)$, the dynamical system 
$(\Omega, \mu, (g_t))$ exhibits Keplerian shear for all absolutely continuous compatible measures $\mu$.
\end{proposition}

\begin{proof}\quad

We use the criterion of Theorem~\ref{thm:ClassiquePrincipal}. It is enough to prove that, 
for all $\xi \in \Z^d \setminus \{0\}$ and all $i \in I$:
\begin{equation*}
A_{\xi, i} := \{v \in \Fcal^r (M, \R^d): \ \Leb (\{d \langle \xi, v_i \rangle =0\} \cap K_i) = 0\}
\end{equation*}
is Baire generic. But $A_{\xi, i}^c = \bigcup_{n \geq 1} \bigcap_{m \geq 1} B_{\xi, i, n, m}$, 
with:
\begin{equation*}
B_{\xi, i, n, m} = \{v \in \Fcal^r (M, \R^d): \ \Leb (\{\norm{d\langle \xi, v_i \circ \varphi_i^{-1} \rangle}{} \leq 1/m\} \cap \varphi_i(K_i)) \geq 1/n\}.
\end{equation*}
All is left is to prove that $\bigcap_{m \geq 1} B_{\xi, i, n, m}$ is meager. Note that:
\begin{equation*}
B_{\xi, i, n, m}^c 
= \left\{v \in \Fcal^r (M, \R^d): \ \Leb (\{ \norm{d\langle \xi, v_i \circ \varphi_i^{-1} \rangle}{} > 1/m \}\cap \varphi_i(K_i)) > \Leb (\varphi_i(K_i))-1/n \right\}.
\end{equation*}
Let $v \in B_{\xi, i, n, m}^c$. By inner regularity of the Lebesgue measure on $\varphi_i (K_i)$, there exists $K' \subset K_i$ compact 
such that $\norm{d\langle \xi, v_i \circ \varphi_i^{-1} \rangle}{} > 1/m$ on $\varphi_i(K')$ and $\Leb(\varphi_i(K'))> \Leb (\varphi_i(K_i))-1/n$. 
By compactness, for all $v'$ close enough to $v$, we have $\norm{d\langle \xi, v_i' \circ \varphi_i^{-1} \rangle}{} > 1/m$ on $\varphi_i(K')$, 
and thus $v' \in B_{\xi, i, n, m}^c$. Hence, each $B_{\xi, i, n, m}$ is closed. We only need to show that the sets 
$\bigcap_{m \geq 1} B_{\xi, i, n, m}$ have empty interior.

\smallskip

Fix $\xi \in \Z^d \setminus \{0\}$, $i \in I$ and $n \geq 1$. Let $\chi_i \in \Ccal^r (V_i,[0,1])$, 
with $\Supp (\chi_i) \subset V_i$ compact and $\chi_i \equiv 1$ on $\varphi_i(K_i)$. For $t \in \R$, let $v (t)$ 
be defined by:
\begin{align*}
v_i (t) \circ \varphi_i^{-1} (x)
& := v_i \circ \varphi_i^{-1} (x)+ t x_1 \chi_i (x) \xi \ \text{ on } (U_i, \varphi_i),\\
v_j (t) \circ \varphi_j^{-1} (x)
& := v_j \circ \varphi_j^{-1} (x) + t x_1 \chi_i \circ \varphi_i \circ \varphi_j^{-1} (x) \mathbf{1}_{U_i \cap U_j} (x) A_{U_i, U_j} (\xi) \ \text{ on } (U_j, \varphi_j), \ j \neq i.
\end{align*}
Then $\lim_{t \to 0} v(t) = v$ in $\Fcal^r (M, \R^d)$. On $\varphi_i(K_i)$, we have $\chi_i \equiv 1$, 
therefore:
\begin{equation*}
d\langle \xi, v_i (t) \circ \varphi_i^{-1} \rangle 
= d\langle \xi, v_i (0) \circ \varphi_i^{-1} \rangle + t \norm{\xi}{}^2 e_1^*,
\end{equation*}
with $e_1^* = (1,0, \ldots, 0)$. By the pigeonhole principle, for all $m \geq 1$, at least one of the functions 
$v(2k/(\norm{\xi}{}^2 m))$, with $0 \leq k \leq \lceil n \Leb (\varphi_i(K_i)) \rceil$, belongs 
to $B_{\xi, i, n, m}^c$. Thus there exists a sequence $(t_m)_{m \geq 1}$ such that 
$v(t_m) \in B_{\xi, i, n, m}^c$ and $\lim_{m \to + \infty} t_m = 0$. This finishes the proof.
\end{proof}

\begin{remark}\quad

If $\Omega = M \times \Tbb^d$ and $r \geq 2$, we can conclude using the (well known, but more difficult to prove) 
fact that a generic function in $\Ccal^r (M, \R)$ is Morse.
\end{remark}

\subsection{Examples}
\label{subsec:Exemples}

The simplest non-trivial example of Keplerian shear is given by the map 
\begin{equation*}
T 
= \left( \begin{array}{cc}
1 & 1 \\
0 & 1
\end{array} \right),
\end{equation*}
acting on $\Tbb^2 = \{(x,y): \ x,y \in \Tbb\}$. This transformation preserves the Lebesgue measure, 
as well as all the circles $\Tbb \times \{y\}$. Keplerian shear is rather easy to 
prove\footnote{This example has been used with some success by the author in a graduate-level exercise course in ergodic theory.}, 
as there is no need to play with charts; one can use directly the Fourier basis on $\Lbb^2 (\Tbb^2, \Leb)$, 
which behaves well under $T$. A slightly more sophisticated version of this argument is 
used in Sub-subsection~\ref{subsubsec:VitesseTransvection} to compute the speed of decay of correlations.

\smallskip

All systems are not that simple. Besides genericity, Theorem~\ref{thm:ClassiquePrincipal} 
provides a useful criterion to prove that a given dynamical system exhibits Keplerian shear. 
We now use it to prove Keplerian shear for two dynamical systems: the billiard in the unit ball $\overline{B}_n \subset \R^n$, 
and the unit speed geodesic flow on $\Tbb^n$ (with the flat metric).

\subsubsection{Billiard in a ball}
\label{subsubsec:Billard}

Let $\overline{B}_n$ be the unit ball in $\R^n$, with $n \geq 2$. Consider a 
particle moving with unit speed in $B_n$, which reflects specularly on the boundary $\Sbb_{n-1}$. 
The phase state is an orbifold $T^1 \overline{B}_n$, and the flow $(g_t)_{t \in \R}$ preserves the Liouville measure 
$\mu_n$ (which here is essentially the Lebesgue measure on $B_n \times \Sbb_{n-1}$).

\begin{proposition}\quad

The dynamical system $(T^1 \overline{B}_n, \mu_n, (g_t)_{t \in \R})$ exhibits Keplerian shear.
\end{proposition}

\begin{proof}\quad

If we exclude trajectories which go through the origin, then any given trajectory lie in the unique 
plane generated by the position and the speed at any given time. Restricted to any such plane, 
the billiard is isomorphic to the billiard in $\overline{B}_2$. Since a disjoint union of systems 
with Keplerian shear still has Keplerian shear, it is enough to prove that 
$(T^1 \overline{B}_2, \mu_2, (g_t)_{t \in \R})$ has Keplerian shear.

\smallskip

The space $T^1 \overline{B}_2$ is $3$-dimensional. The angle $\theta \in (-\pi/2, \pi/2)$ 
with which the trajectories hit the boundary is an invariant of the flow. Hence, 
$(T^1 \overline{B}_2, \mu_2, (g_t)_{t \in \R})$ is isomorphic to 
$(\Omega, \tilde{\mu}, (\tilde{g}_t)_{t \in \R})$, where:
\begin{itemize}
\item $\Omega = (-\pi/, \pi/2) \times \Tbb^2$;
\item $\tilde{\mu} = 2^{-1}\cos(\theta) \de \theta \otimes \Leb_{\Tbb^2}$;
\item $\tilde{g}_t (\theta,x) = (\theta, x+tv(\theta))$,
\end{itemize}
and $v (\theta) = 2 \cos (\theta) (1, 1/2-\theta/\pi)$. In particular,
\begin{equation*}
v' (\theta) 
= - 2 \sin (\theta) \left( \begin{array}{c} 1 \\ \frac{1}{2}-\frac{\theta}{\pi} + \frac{1}{\pi} \cot (\theta) \end{array}\right).
\end{equation*}
For all $\xi \in \Z^2 \setminus \{0\}$, the function $\langle \xi, v' \rangle$ is analytic and non-zero, and thus 
its zero set is discrete. By Theorem~\ref{thm:ClassiquePrincipal}, the system $(T^1 \overline{B}_2, \mu_2, (g_t)_{t \in \R})$ 
has Keplerian shear.
\end{proof}

A similar proof applies to the billiard in an ellipsoid, or the geodesic flow on an ellipsoid.

\subsubsection{Geodesic flow on the torus}
\label{subsubsec:FlotTM}

The second example we discuss is the unit speed geodesic flow on the torus $\Tbb^n$, 
with $n \geq 1$. This flow, again, preserves the Liouville measure. 

\begin{proposition}\quad
\label{prop:FlotToreCisaillement}

The dynamical system $(T^1 \Tbb^n, \Liouv, (g_t)_{t \in \R})$ exhibits Keplerian shear.
\end{proposition}

\begin{proof}\quad

The manifold $T^1 \Tbb^n$ is trivializable, and thus isomorphic to $\Tbb^n\times \Sbb_{n-1}$. 
The geodesic flow $(g_t)_{t \in \R}$ acts on $T^1 \Tbb^n$ by:
\begin{equation*}
g_t (x,v) 
= (x+tv, v).
\end{equation*}
Let $\xi \in \Z^n \setminus \{0\}$. Then $\de \langle \xi, v \rangle$ vanishes at only two points, 
which are $\pm \xi / \norm{\xi}{}$. By Theorem~\ref{thm:ClassiquePrincipal}, the system 
$(T^1 \Tbb^n, \Liouv, (g_t)_{t \in \R})$ has Keplerian shear.
\end{proof}

\subsection{Unique ergodicity}
\label{subsec:UniqueErgodicite}

In this subsection, we describe the relation between Keplerian shear and the unique ergodicity 
of a transformation acting on spaces of probability measures, as introduced by F.~Maucourant~\cite{Maucourant:2016}. 
We drop the assumption that the function $v$ generating the flow be $\Ccal^1$: 
here, continuity is enough.

\subsubsection{Definition and relation with Keplerian shear}

Let $\pi : \Omega \to M$ be a compact affine tori bundle, $(g_t)$ a compatible flow on $\Omega$, 
and $\nu \in \Pcal (M)$. Denote by $\Pcal_\nu \subset \Pcal (\Omega)$ the subspace 
of probability measures $\tilde{\mu}$ such that $\pi_* \tilde{\mu} = \nu$, 
and by $\nu \otimes \Leb$ the unique compatible measure on $\Omega$ such that $\pi_* (\nu \otimes \Leb) = \nu$.

\smallskip

Let $G_t := g_{t,*}$ act continuously on $\Pcal (\Omega)$, which is compact when endowed with the weak convergence. 
Since the flow is compatible, $(G_t)$ preserves $\Pcal_\nu$, which is also compact. Note that $\nu \otimes \Leb$ is a 
fixed point of $(G_t)$, so $\delta_{\nu \otimes \Leb}$ is $(G_t)$-invariant.

\begin{theorem}\quad
\label{thm:UniqueErgodicite}

Let $\pi : \Omega \to M$ be a compact affine tori bundle. Let $(g_t)$ be a compatible flow on $\Omega$. 
Let $\nu \in \Pcal (M)$.

\smallskip

The system $(\Omega, \nu \otimes \Leb, (g_t))$ exhibits Keplerian shear if and only if 
$G_t (\mu) \to \nu \otimes \Leb$ for all $\mu \in \Pcal_\nu$. Then $(\Pcal_\nu, (G_t))$ is uniquely ergodic.
\end{theorem}

\begin{proof}\quad

Let $\pi : \Omega \to M$, $(g_t)$ and $\nu$ be as in the hypotheses of the theorem. 
First, we assume that $(\Omega, \nu \otimes \Leb, (g_t))$ exhibits Keplerian shear.
We can find a countable cover of $M$ by disjoint open charts $(U_i)_{i \in I}$, 
up to a $\nu$-negligible subset. Then all $(U_i \times \Tbb^d, \nu_{|U_i} \otimes \Leb, (g_t))$ 
exhibit Keplerian shear. 


\smallskip

Let $\mu$ be in $\Mcal (U_i \times \Tbb^d)$ with $\pi_* \mu = \nu_{|U_i}$. 
Endow $U_i$ with any bounded Riemannian metric, and $\Tbb^d$ with a flat metric. This yields 
a Riemannian metric on $U_i \times \Tbb^d$ (e.g.\ the product metric), from which we 
get a Wasserstein distance $d_W$, which metrizes the weak convergence.

\smallskip

We denote by $*$ the fiberwise convolution on each torus. Fix $\varepsilon >0$, and let 
$\rho_\varepsilon$ be an absolutely continuous measure supported on 
$\overline{B}_{\Tbb^d} (0, \varepsilon)$. Then $d_W (\rho_\varepsilon, \delta_0) \leq \varepsilon$, 
whence, for all $t$:
\begin {equation*}
d_W (G_t (\mu * \rho_\varepsilon), G_t (\mu)) 
= d_W (\mu * G_t (\rho_\varepsilon), \mu * G_t (\delta_0)) 
\leq \varepsilon.
\end {equation*}
On the other hand, $\mu * \rho_\varepsilon \ll \nu_{|U_i} \otimes \Leb$ and $\pi_* (\mu * \rho_\varepsilon) = \nu$. 
As we see by integrating against test functions, Keplerian shear implies that 
$G_t (\mu * \rho_\varepsilon) \to \nu_{|U_i} \otimes \Leb$ weakly. In particular, 
$d_W (G_t (\mu * \rho_\varepsilon), \nu_{|U_i} \otimes \Leb) \leq \varepsilon$ for all large enough 
$t$, whence $d_W (G_t(\mu),\nu_{U_i} \otimes \Leb) \leq 2 \varepsilon$. As this is true for all $\varepsilon >0$, 
we get $G_t(\mu) \to \nu_{|U_i} \otimes \Leb$. Since this is true for all $i$, 
$G_t (\mu) \to \nu \otimes \Leb$ for all $\mu \in \Pcal_\nu$. Hence, $(\Pcal_\nu, (G_t))$ is uniquely ergodic.

\smallskip

Assume now that $G_t (\mu) \to \nu \otimes \Leb$ for all $\mu \in \Pcal_\nu$. By~\cite[Theorem~1]{Maucourant:2016}, 
$g_1$ is asynchronuous, so the set of points $x$ of $M$ such that $(g_t)$ acts on $\{x\} \times \Tbb^d$ by 
an irrational translation has full $\nu$-measure. Hence, the invariant $\sigma$-algebra is 
$\pi^* \Bcal_M$.

\smallskip

Let $(U_i)_{i \in I}$ be an open cover of $M$ by charts. Let $f \in \Ccal (\Omega, \C)$. 
Let $i \in I$ and $\rho (x,y) = a(x) b(y)$ on $U_i \times \Tbb^d$, for $a \in \Ccal_c (U_i, \R_+^*)$ 
and $b \in \Ccal (\Tbb^d, \R_+^*)$ such that $\int_{\Tbb^d} b \dd \Leb = 1$. Take $\rho \equiv 0$ on $\pi^{-1}(U_i^c)$ 
and $a \equiv 0$ on $U_i^c$. 
Let $\mu$ be the probability measure on $\Omega$ defined by $\mu_{|\pi^{-1}(U_i)} := \nu_{U_i} \otimes (b \dd \Leb)$ and 
$\mu_{|\pi^{-1}(U_i)^c} := \nu_{U_i^c} \otimes \Leb$. Then $\mu \in \Pcal_\nu$, and, for all $t$:
\begin{equation*}
\int_\Omega f \circ g_t \cdot \rho \dd \nu \otimes \Leb 
= \int_{U_i} fa \cdot g_{t,*} (\nu_{|U_i} \otimes b \dd \Leb) 
= \int_\Omega fa \cdot G_t (\mu).
\end{equation*}
By assumption, $G_t (\mu)$ converges weakly to $\nu \otimes \Leb$, so the quantity above 
converges to:
\begin{equation*}
\int_\Omega fa \dd \nu \otimes \Leb 
= \Ebb_{\nu \otimes \Leb} (\Ebb_{\nu \otimes \Leb} (f|\Ical) \Ebb_{\nu \otimes \Leb} (\rho|\Ical)).
\end{equation*}
By Lemma~\ref{lem:LemmeAnalyseFonctionnelle}, $(\Omega, \nu \otimes \Leb, (g_t))$ exhibits Keplerian shear.
\end{proof}

\begin{remark}[Keplerian shear is stronger than unique ergodicity]\quad

F.~Maucourant gives an example~\cite{Maucourant:2016} of a compatible flow and a measure $\nu$ such that $(\Pcal_\nu, (G_t))$ 
is uniquely ergodic, but the fixed point $\nu \otimes \Leb$ behaves like an indifferent fixed point: 
there are exceptional sequences of times $(t_i)$ for which $G_{t_i} (\nu \otimes \delta_0)$ is 
far from $\nu \otimes \Leb$. As a corollary, the unique ergodicity of $(\Pcal_\nu, (G_t))$ does not imply 
that $(\Omega, \nu \otimes \Leb, (g_t))$ has Keplerian shear.
\end{remark}

\subsubsection{An application : Gauss' circle problem}
\label{subsubsec:Gauss}

The alternative characterization of Keplerian shear given by Theorem~\ref{thm:UniqueErgodicite} is also useful 
in settings which use non-absolutely continuous measures. Let us give an elementary application to a variation on 
Gauss' circle problem. Let $S(x,r)$ be the sphere of center $x$ and radius $r$ in $\R^n$, with $n \geq 2$. 
Let $\varepsilon \in (0, 1/2)$. What is the number of integer points in an $\varepsilon$-neighborhood of $S(x,r)$?

\smallskip

Let $\sigma_{x,r}$ be the uniform measure on $S (x,r)$, and $\varpi$ the canonical projection from $\R^n$ to $\Tbb^n$.
Take $\Omega := \Sbb_{n-1} \times \Tbb^n$, with $g_t (v,y) = (v,y+tv)$ and 
$\nu$ the uniform measure on $\Sbb_{n-1}$. Let $f(y) := \mathbf{1}_{|y| \leq \varepsilon}$ on $\Tbb^n$. 
Then:
\begin{equation*}
\sigma_{x,r} (\{y \in \R^n: \ d(y,\Z^n) \leq \varepsilon\}) 
= (\varpi_*\sigma_{x,r}) (\{y\in \Tbb^n: \ d(y,0) \leq \varepsilon\}) 
= G_t (\nu \otimes \delta_{\varpi (x)}) (f).
\end{equation*}
The system $(\Omega, \nu \otimes \Leb, (g_t))$ has Keplerian shear by Proposition~\ref{prop:FlotToreCisaillement}, 
so that:
\begin{equation*}
\lim_{r \to + \infty} \sigma_{x,r} (\{y \in \R^n: \ d(y,\Z^n) \leq \varepsilon\}) 
= \Leb (B_{\R^n} (0, \varepsilon)) 
= \varepsilon^n \Leb (B_{\R^n} (0, 1).
\end{equation*}

In addition, $S(x,r) \cap\overline{B}(\Z^n, \varepsilon)$ consists of finitely many caps, which get 
flatter and flatter as $r$ increases; the number of integer points $\varepsilon$-close to $S(x,r)$ 
is the number of such caps. Let us direct there caps by the outward normal at their center. Since 
the measure supported by the projection on $\Sbb_{n-1} \times \Tbb^n$ of these caps equidistributes 
in $\Sbb_{n-1} \times B(0,\varepsilon)$, we get that the average area (for $\varpi_* \sigma_{x,r}$) 
of each cap converges to:
\begin{equation*}
\frac{\text{Average cross-section of } B_{\R^n}(0,\varepsilon)}{\Leb_{n-1} (S(0,r))} 
= \frac{\varepsilon^{n-1} \Leb (B_{\R^n}(0,1))}{2 r^{n-1} \Leb_{n-1} (\Sbb_{n-1})}.
\end{equation*}
Hence, the number of integer points in an $\varepsilon$-neighborhood of $S(x,r)$ converges, as $r$ goes to infinity, 
to:
\begin{equation*}
\varepsilon^n \Leb (B_{\R^n} (0, 1)) \cdot \frac{2 r^{n-1} \Leb_{n-1} (\Sbb_{n-1})}{\varepsilon^{n-1} \Leb (B_{\R^n}(0,1))} 
= 2 \varepsilon r^{n-1} \Leb_{n-1} (\Sbb_{n-1}).
\end{equation*}

This stays true if the sphere is replaced by any compact manifold, under non-resonancy conditions 
which ensure Keplerian shear for the relevant dynamical system. Note also that for the sphere, 
by integrating over $r$, one recovers the more elementary fact that the number of integral points at 
distance $r$ from the origin is equivalent to $r^n \Leb (B_{\R^n}(0,1))$.

\smallskip

This result is not optimal. For instance, the best known bounds for Gauss' circle problem~\cite{Huxley:2003} 
imply that:
\begin{equation*}
\Card [\Z^2 \cap (S(0,r)+B(0,\varepsilon)) ] 
\sim 2 \varepsilon r^{n-1} \Leb_{n-1} (\Sbb_{n-1}) + O (r^\frac{131}{208} \ln (r)^{\frac{18627}{8320}}),
\end{equation*}
and this error bound holds if the circle is replaced by a closed $\Ccal^3$ curve with non-vanishing curvature. 
The proof of this result, however, requires more technology\footnote{Typically, it uses a decomposition of the circle 
into ``big arcs'' and ``small arcs'', which can also be used to prove Keplerian shear directly without using the Fourier transform.}.

\subsection{Speed of mixing}
\label{subsec:Speed}

Keplerian shear is a qualitative property of a measure-preserving dynamical system, which asserts 
the convergence to zero on average of the conditional correlations:
\begin{equation*}
\Ebb (\Cov_t (f_1, f_2|\Ical)) 
= \Ebb ( \overline{f}_1 \cdot f_2 \circ g_t) - \Ebb( \Ebb(\overline{f}_1|\Ical)  \Ebb(f_2|\Ical) ).
\end{equation*}
As with the notion of mixing, one cannot expect a rate of convergence for all observables $f_1$, $f_2 \in \Lbb^2$. 
However, we may get a rate of convergence if $f_1$ and $f_2$ are regular enough. We may also need 
assumptions of the measure $\mu$ and the critical points of the functions $\langle \xi, v \rangle$. 

\smallskip

In the examples we discuss below, $f_1$ and $f_2$ shall belong to \textit{anisotropic Sobolev spaces} 
(or, more precisely, weighted anisotropic Sobolev spaces). The regularity of such observables depends 
on the direction. We refer the reader to the monography by H.~Triebel for additional 
information~\cite[Chapters~5-6]{Triebel:2006}\footnote{A small difference is that our spaces $H^{s,0}$ and 
$H^{s,\frac{n-1}{2}}$ below do not fit exactly in the framework of Triebel, because the weights 
do not satisfy the assumptions at the beginning of~\cite[Chapters~6]{Triebel:2006}. However, 
one can write for instance $H^{s,0} (\Tbb^2) = \Lbb^2 (\Tbb^1) \oplus \Hcal^{s,0} (\Tbb^2)$, 
where $\Lbb^2 (\Tbb^1)$ has no effect on the correlations and $\Hcal^{s,0} (\Tbb^2)$ fits into 
Triebel's framework.}.

\smallskip

In our setting, we need relatively little regularity in the direction of the invariant tori: 
what matters most is the regularity transversaly to the invariant tori. This is not surprising in view of 
Theorem~\ref{thm:UniqueErgodicite}, which asserts roughly that $\Ebb (\Cov_t (f_1, f_2|\Ical))$ vanishes, 
where $f_1$ is Lipschitz and $f_2$ is e.g.\ $\Leb \otimes \delta_0$ on $M \times \Tbb^d$. 
In this case, $f_2$ is a distribution which is more regular transversaly to the invariant tori than in the 
direction of the invariant tori.

\smallskip

Instead of working out a general statement, we discuss two simple systems: the parabolic automorphism of $\Tbb^2$ 
at the beginning of Subsection~\ref{subsec:Exemples}, and the unit speed geodesic flow on $\Tbb^n$.

\subsubsection{Transvection on $\Tbb^2$}
\label{subsubsec:VitesseTransvection}

Consider the map 
\begin{equation*}
T 
= \left( \begin{array}{cc}
1 & 0 \\
1 & 1
\end{array} \right),
\end{equation*}
acting on $\Tbb^2$, endowed with the Lebesgue measure. Let us define suitable anisotropic Sobolev spaces. For $\xi \in \R^2$, let:
\begin{equation*}
h (\xi) 
:= \left\{ 
\begin{array}{lll}
\left( 1+\frac{\xi_1^2}{\xi_2^2} \right)^{\frac{1}{2}} & \text{ if } & \xi_2 \neq 0 \\
1 & \text{ if } & \xi_2 = 0
\end{array}
\right. .
\end{equation*}
For any real number $s \geq 0$, let:
\begin{equation*}
H^{s,0} (\Tbb^2) 
:= \left\{ f \in \Lbb^2 (\Tbb^2): \ \norm{f}{H^{s,0} (\Tbb^2)}^2 := \sum_{\xi \in 2\pi\Z^2} h(\xi)^{2s} |\hat{f}|^2 (\xi) <+\infty \right\}.
\end{equation*}
The following proposition gives decay bounds on the correlation coefficients for Sobolev or analytic observables.

\begin{proposition}\quad
\label{prop:VitesseTransvection}

Let $f_1$, $f_2$ be in $H^{s,0} (\Tbb^2,\R)$. Then:
\begin{equation*}
| \Ebb (\Cov_n (f_1, f_2|\Ical)) |
\leq \frac{4^s}{n^{2s}} \norm{f_1}{H^{s,0} (\Tbb^2)} \norm{f_2}{H^{s,0} (\Tbb^2)}.
\end{equation*}

If $f_1$ and $f_2$ are analytic, then there exist constants $c$, $C>0$ (depending on $f_1$ and $f_2$) such that, 
for all $n \in \Z$,
\begin{equation*}
| \Ebb (\Cov_n (f_1, f_2|\Ical)) |
\leq C e^{-c|n|}.
\end{equation*}
\end{proposition}

\begin{proof}\quad

Let $f_1$, $f_2$ be in $H^{s,0} (\Tbb^2)$. By Plancherel's theorem,
\begin{equation*}
\Ebb ( f_1  \cdot f_2 \circ T^n) 
= \sum_{\xi \in 2\pi\Z^2} \overline{\hat{f}}_1 (\xi) \hat{f}_2 (T^{*n} \xi),
\end{equation*}
so that:
\begin{align*}
| \Ebb (\Cov_n (f_1, f_2|\Ical)) |
& = \left| \sum_{\substack{\xi \in 2\pi\Z^2 \\ \xi_2 \neq 0}} \overline{\hat{f}}_1 (\xi) \hat{f}_2 (T^{*n} \xi) \right| \\
& \leq \sum_{\substack{\xi \in 2\pi\Z^2 \\ \xi_2 \neq 0}} \left[ |\overline{\hat{f}}_1 h^s| \cdot |\hat{f}_2 h^s| \circ T^{*n} \cdot h^{-s} \cdot h^{-s} \circ T^{*n} \right] (\xi) \\
& \leq \norm{f_1}{H^{s,0} (\Tbb^2)} \norm{f_2}{H^{s,0} (\Tbb^2)} \sup_{\substack{\xi \in 2\pi\Z^2 \\ \xi_2 \neq 0}} \{ h^{-s} (\xi) h^{-s} (T^{*n} \xi) \}.
\end{align*}

Let $\xi_2 \in 2\pi\Z \setminus \{0\}$. The function $\xi_1 \mapsto h^{-s} (\xi_1, \xi_2) h^{-s} (\xi_1+n \xi_2, \xi_2)$ is maximal 
for $\xi_1 = -n \xi_2 /2$, where its value is $(1+n^2/4)^{-s}$, so that:
\begin{equation*}
| \Ebb (\Cov_n (f_1, f_2|\Ical)) |
\leq \frac{4^s}{n^{2s}} \norm{f_1}{H^{s,0} (\Tbb^2)} \norm{f_2}{H^{s,0} (\Tbb^2)}.
\end{equation*}

The proof for analytic functions is essentially the same. The only remark needed is that, 
if $f$ is analytic on the torus, then there exist constants $c'$, $C'>0$ such that 
$|\hat{f}| (\xi) \leq C' e^{-c' |\xi|}$.
\end{proof}

The map $T$ is especially well-behaved: not only does it acts nicely on Fourier series, 
but its shearing (the derivative of $v$) does not vanish. The estimates of Proposition~\ref{prop:VitesseTransvection} 
are thus a best case behaviour, that we do not expect to hold for more general systems.

\subsubsection{Speed for the geodesic flow on the torus}
\label{subsubsec:VitesseTore}

The geodesic flow is harder to analyse than the previous example: not only does it lack 
its algebraic structure, but the functions $\langle \xi, v \rangle$ have vanishing gradient 
at two points for any non-zero $\xi$. Hence, we cannot expect the same rate of convergence. We use the stationary phase 
method to compute the speed of convergence. This yields a polynomial rate of decay for a large 
space of observables belonging again to some anisotropic Sobolev spaces (Proposition~\ref{prop:VitesseFlotGeodesique}).

\smallskip


The definition of these anisotropic Sobolev spaces is however slightly more delicate. Let $n \geq 2$ 
and $s>(n-1)/2$. For $(k, \xi) \in \R^{n-1} \times 2 \pi \Z^n$, let:
\begin{equation*}
h (k,\xi) 
:= \left\{ 
\begin{array}{lll}
1+\frac{(1+\norm{k}{}^2)^{\frac{s}{2}}}{\norm{\xi}{}^{\frac{n-1}{4}}} & \text{ if } & \xi \neq 0 \\
1 & \text{ if } & \xi = 0
\end{array}
\right. .
\end{equation*}

We see $T^1 \Tbb^n$ as $\Sbb_{n-1} \times \Tbb^n$. Fix a finite open cover by charts 
$(U_i, \varphi_i)$ of $\Sbb_{n-1}$, and a smooth partition of the unit $(\chi_i)$ 
subordinated to $(U_i)$. Then define:
\begin{equation*}
H^{s,\frac{n-1}{2}} (\Sbb_{n-1} \times \Tbb^n) 
:= \left\{ f \in \Lbb^2 : \  \sum_i \sum_{\xi \in 2\pi\Z^2} \int_{\R^{n-1}} h^2 |\widehat{[(f \chi_i) \circ (\varphi_i^{-1},\id)]}|^2 (x,\xi) \dd x <+\infty \right\},
\end{equation*}
and denote by $\norm{\cdot}{H^{s,\frac{n-1}{2}}}^2$ the norm appearing in this definition. 
In the same way, we define the Sobolev space $H^s (\Sbb_{n-1})$. These spaces do not depend on the 
choice of the family of charts and of the partition of the unit.

\smallskip

The following proposition gives decay bounds on the correlation coefficients for observables in $H^{s,\frac{n-1}{2}}$.

\begin{proposition}\quad
\label{prop:VitesseFlotGeodesique}

Let $n \geq 2$ and $s > (n-1)/2$. There exists a constant $C$ such that, for all $f_1$, $f_2 \in H^{s,\frac{n-1}{2}} (\Sbb_{n-1} \times \Tbb^n)$,
\begin{equation}
\label{eq:VitesseFlotGeodesique}
| \Ebb (\Cov_t (f_1, f_2|\Ical)) |
\leq \frac{C}{t^{\frac{n-1}{2}}} \norm{f_1}{H^{s,\frac{n-1}{2}}} \norm{f_2}{H^{s,\frac{n-1}{2}}}.
\end{equation}
\end{proposition}

\begin{proof}\quad

In this proof, the letter $C$ shall denote a constant which may change from line to line, but which depends only 
on the dimension $n$ and on the parameter $s$.

\smallskip

Let $s > (n-1)/2$. Let $f_1$, $f_2$ be in $\Ccal^\infty (\Sbb_{n-1} \times \Tbb^n, \C)$. Denote by $f_i^\xi (x)$ 
the Fourier transform of $f_i(x, \cdot)$ evaluated in $\xi \in 2\pi\Z^n$. 
By Plancherel's and Fubini-Lebesgue theorems, the conditional covariance is equal to:
\begin{equation*}
\Ebb (\Cov_n (f_1, f_2|\Ical))
= \sum_{\substack{\xi \in 2\pi\Z^n \\ \xi \neq 0}} \int_{\Sbb_{n-1}} \overline{f_1^\xi} (x) f_2^\xi (x) e^{i t \langle \xi, x \rangle} \dd x.
\end{equation*}

Let $\chi \in \Ccal^\infty (\Sbb_{n-1}, [0,1])$ be such that $\chi \equiv 1$ near $N:= (1,0, \ldots, 0)$ and 
$\chi (-x) = 1-\chi (x)$. Let $\varphi_+ : \Sbb_{n-1} \setminus \{S\} \to \R^{n-1}$ (resp. $\varphi_- : \Sbb_{n-1} \setminus \{N\} \to \R^{n-1}$) 
be the stereographic projection from the North (resp. South) pole. Let $\xi \in 2\pi\Z^n$, 
and $R_\xi$ a rotation which send $\xi/\norm{\xi}{}$ to $N$. Finally, let $\psi_{\xi,\pm} := (\varphi_{\pm} \circ R_\xi)^{-1}$. Then:
\begin{align*}
\int_{\Sbb_{n-1}} \overline{f_1^\xi} (x) f_2^\xi (x) e^{i t \langle \xi, x \rangle} \dd x 
& = \int_{\R^{n-1}} \left( \overline{f_1^\xi} f_2^\xi \right) \circ \psi_{\xi,+} (x) e^{i t \langle \xi, \psi_{\xi,+} (x) \rangle} \frac{\chi \circ \varphi_+^{-1} (x)}{\Jac (\varphi_+^{-1}) (x)} \dd x \\
& \hspace{2em} + \int_{\R^{n-1}} \left( \overline{f_1^\xi} f_2^\xi \right) \circ \psi_{\xi,-} (x) e^{i t \langle \xi, \psi_{\xi,-} (x) \rangle} \frac{(1-\chi) \circ \varphi_-^{-1} (x)}{\Jac (\varphi_-^{-1}) (x)} \dd x.
\end{align*}
The function $1/\Jac (\varphi_\pm^{-1})$ is in $\Ccal_b^\infty (\R^{n-1})$ and the function $x \mapsto \langle \xi, \psi_{\xi,\pm} (x) \rangle$ has 
a unique critical point in $0$ which is non-degenerate. By the stationary phase method~\cite[Chapter~7.7]{Hormander:1983}, there exists a constant $C$ such that:
\begin{align*}
\left| \int_{\Sbb_{n-1}} \overline{f_1^\xi} (x) f_2^\xi (x) e^{i t \langle \xi, x \rangle} \dd x \right|
& \leq \frac{C}{(t \norm{\xi}{})^{\frac{n-1}{2}}} \left[ \int_{\R^{n-1}} \left| \widehat{ \left(\overline{f_1^\xi} f_2^\xi\right) \circ \psi_{\xi, +} \cdot \chi \circ \varphi_+^{-1}} \right| (k) \dd k \right. \\
& \hspace{2em} + \left. \int_{\R^{n-1}} \left| \widehat{ \left(\overline{f_1^\xi} f_2^\xi\right) \circ \psi_{\xi, -} \cdot (1-\chi) \circ \varphi_-^{-1} } \right| (k) \dd k \right] \\
& \leq \frac{C}{(t \norm{\xi}{})^{\frac{n-1}{2}}} \norm{f_1^\xi}{H^s (\Sbb_{n-1})} \norm{f_2^\xi}{H^s (\Sbb_{n-1})},
\end{align*}
where we used the fact that $\norm{\widehat{fg}}{\Lbb^1 (\R^{n-1})} \leq C \norm{f}{H^s (\R^{n-1})} \norm{g}{H^s (\R^{n-1})}$ whenever $s > (n-1)/2$. Hence:
\begin{align*}
\Ebb (\Cov_n (f_1, f_2|\Ical))
& \leq \frac{C}{t^{\frac{n-1}{2}}} \sum_{\substack{\xi \in 2\pi\Z^n \\ \xi \neq 0}} \frac{\norm{f_1^\xi}{H^s (\Sbb_{n-1})} \norm{f_2^\xi}{H^s (\Sbb_{n-1})}}{\norm{\xi}{}^{\frac{n-1}{2}}} \\
& \leq \frac{C}{t^{\frac{n-1}{2}}} \sqrt{\sum_{\substack{\xi \in 2\pi\Z^n \\ \xi \neq 0}} \frac{\norm{f_1^\xi}{H^s (\Sbb_{n-1})}^2 }{\norm{\xi}{}^{\frac{n-1}{2}}}} \sqrt{\sum_{\substack{\xi \in 2\pi\Z^n \\ \xi \neq 0}} \frac{\norm{f_2^\xi}{H^s (\Sbb_{n-1})}^2}{\norm{\xi}{}^{\frac{n-1}{2}}}}.
\end{align*}

Finally, using our local charts $(U_i, \varphi_i)$ on $\Sbb_{n-1}$:
\begin{align*}
\sum_{\substack{\xi \in 2\pi\Z^n \\ \xi \neq 0}} \frac{\norm{f_1^\xi}{H^s (\Sbb_{n-1})}^2 }{\norm{\xi}{}^{\frac{n-1}{2}}} 
& \leq C \sum_i \sum_{\substack{\xi \in 2\pi\Z^n \\ \xi \neq 0}} \int_{\R^{n-1}} \frac{(1+\norm{k}{}^2)^s}{\norm{\xi}{}^{\frac{n-1}{2}}} |\widehat{[(f_1\chi_i) \circ (\varphi_i^{-1}, \id)]}|^2 (x, \xi) \dd x \\
& \leq C \norm{f_1}{H^{s,\frac{n-1}{2}} (\Sbb_{n-1} \times \Tbb^n)}^2.
\end{align*}

That finishes the proof for smooth observables $f_1$ and $f_2$. But, for fixed $t$, the 
correlation function $\Ebb(\Cov_t(\cdot,\cdot|\Ical))$ is bilinear and continuous from $\Lbb^2$ to $\C$. Since the 
$H^{s,\frac{n-1}{2}}$ norm is stronger than the $\Lbb^2$ norm, $\Ebb(\Cov_t(|\Ical))$ is also continuous 
from $H^{s,\frac{n-1}{2}}$ to $\C$. But $\Ccal^\infty$ is dense in $H^{s,\frac{n-1}{2}}$, 
so the bound~\eqref{eq:VitesseFlotGeodesique} actually holds for any two observables in $H^{s,\frac{n-1}{2}}$.
\end{proof}

Assuming that the observables $f_1$ and $f_2$ have higher regularity, 
standard formulations of the stationary phase method yield a higher order development of 
$\Ebb(\Cov_t (f_1, f_2|\Ical))$ as $t$ goes to infinity.

\smallskip

Assume now that we change the flow on $\Sbb_{n-1} \times \Tbb^n$, for instance by making the velocity depend on the direction. 
Then the rates we got in Proposition~\ref{prop:VitesseFlotGeodesique} may not be generic. We shall sketch the difficulties 
encountered with more general systems. Let $n \geq 3$ and $M$ be a compact connected $(n-1)$-dimensional smooth 
manifold, and let $v : M \to \R^n$ be smooth. Consider the flow $g_t (x,y)=(x,y+tv(x))$ 
on $M \times \Tbb^n$. If $Dv$ is never degenerate (which is a $\Ccal^1$-open condition on $v$), 
then $v$ is an immersion. If in addition the extrinsic curvature of the immersed manifold is never degenerate, 
then we get rates of convergence as in Proposition~\ref{prop:VitesseFlotGeodesique}. 
However, if the extrinsic curvature is never degenerate, then the Gauss map $M \to \Sbb_{n-1}$ 
is a local diffeomorphism, so a diffeomorphism (since $n \geq 3$), and thus $M$ is a sphere. 

\smallskip

In other words, if $M$ is not a sphere, then we have to deal with degenerescences of 
the extrinsic curvature of $v(M)$. If such a degenerescence happens in a rational direction of $\R^n$, 
then we would get a speed of convergence in $O (t^{-\frac{n-1-r}{2}})$, where $r$ is the 
corank of the Hessian in the given direction. 
If this degenerescence happens in a direction $u$ which is not rational, then this bound could be improved, 
although any improvement would depend on the Diophantine properties of $u$ (the bound getting better if $u$ 
is badly approximable by rationals). In particular, one cannot hope to get a significantly better bound 
than $O (t^{-\frac{n-1-r}{2}})$ in a Baire generic setting, as Baire generic directions are Liouville.

\smallskip

For $n \geq 2$, the same kind of obstruction may happen for $v : \Sbb_{n-1} \to \R^n$. 
For a $\Ccal^3$-open set of such functions $v$, the map $v$ has non-degenerate inflexion points. 
Without further argument about the directions these inflexion points occur, this would 
for instance yield a rate of decay of only $O (t^{-\frac{1}{3}})$ if $n=2$.

\section{Stretched Birkhoff sums}
\label{sec:StretchedBirkhoff}

We present in this sub-section another class of systems which may exhibit Keplerian shear. 
The examples of Subsection~\ref{subsec:TheoremePrincipal} are based on translations on the torus, 
which are a family of non-mixing dynamical systems. In this section, the elementary brick will be given by 
suspension flows with constant roof function. The family of examples we get 
includes many non-Hamiltonian systems.

\smallskip

Let $(A, \nu, T)$ be a measure-preserving dynamical system. 
For $v > 0$, the suspension flow with constant roof $1$ and speed $v$ is the measure-preserving semi-flow 
$(\widetilde{A}, \tilde{\nu}, (g_t^v)_{t \geq 0})$ defined by:
\begin{itemize}
\item $\widetilde{A} := (A \times [0, 1])_{(x,1) \sim (T(x),0)}$;
\item $g_t^v [(x,s)] = [(x,s+vt)]$;
\item $\tilde{\nu} := \nu \otimes \Leb$ on the fundamental domain $A \times [0,1)$.
\end{itemize}
Such a suspension flow is ergodic, but cannot be mixing, as it has the rotation on the circle 
as a factor.

\smallskip

Now, we give ourselves:
\begin{itemize}
\item a $n$-dimensional $\Ccal^1$ manifold $M$, with $n \geq 1$;
\item a measure-preserving ergodic dynamical system $(A, \nu, T)$;
\item a measurable function $v : M \to \R_+^*$.
\end{itemize}
With this data we construct a new semi-flow $(\Omega, (g_t)_{t \geq 0})$ 
with $\Omega := \widetilde{A} \times M$ and $g_t (x,y) := (g_t^{v(y)} (x), y)$. 
A measure $\mu \in \Pcal (\Omega)$ is said to be \textit{compatible} if 
it is equal to $\tilde{\nu} \otimes \tilde{\mu}$ for some $\tilde{\mu} \in \Pcal (M)$. 
Compatible measures are preserved by $(g_t)$.

\smallskip

If $(A, \nu, T)$ is invertible, the suspension semi-flow can be extended 
to a flow, in which case $v$ may take negative values. The following theorem 
also holds in this alternative setting.

\begin{theorem}\quad
\label{thm:CisaillementSuspensions}

Let $(\Omega, (g_t)_{t \geq 0})$ be a system defined as above, with $v \in \Ccal^1 (M, \R_+^*)$. 
Let $\mu$ be an absolutely continuous compatible measure.
If $\Leb (dv=0)=0$, then $(\Omega, \mu, (g_t)_{t \geq 0})$ exhibits Keplerian shear.
\end{theorem}

\begin{proof}\quad

Let $\Ical_A$ be the invariant $\sigma$-algebra of $(A, T)$, and $\Bcal_M$ the Borel $\sigma$-algebra of $M$. 
As a measured space, we can see $\Omega$ as $A \times \Sbb_1 \times M$. Up to completion with respect to $\mu$, 
the invariant $\sigma$-algebra of $(\Omega, (g_t)_{t \geq 0})$ is $\Ical := \Ical_A \otimes \{\emptyset, \Sbb_1\} \otimes \Bcal_M$.

\smallskip

Let $U := \{dv \neq 0\} \subset M$. Let $(U_i, \psi_i)_{i \in I}$ be a countable cover of $U$ by charts, 
with $\varphi_i : U_i \to W_i' \subset \R^n$ and $W_i'$ bounded. Using the local normal form of submersions, 
we assume that $v \circ \varphi_i^{-1} (z) = z_1>0$. We write $z' = (z_2, \ldots, z_n)$. Let $(V_i)_{i \in I}$ 
be a partition of $U$ by open sets, up to a Lebesgue negligible subset of $U$, such that $\overline{V}_i \subset U_i$ for all $i$. 
We write $W_i := \varphi_i (V_i)$.

\smallskip

We apply Lemma~\ref{lem:LemmeAnalyseFonctionnelle}, with the Banach space $\Bcal = \Bcal^* = \Lbb^2 (\Omega, \mu)$, and:
\begin{align*}
E & = E^* = \\
& \bigcup_{i \in I} \left\{f(x,y,z)=a(x)e^{i \xi y} b (\varphi_i (z)_1) c (\varphi_i (z)'): \ a \in \Lbb^2 (A, \nu), \ \xi \in 2\pi\Z, b, c \in \Ccal_c^1, bc \in \Ccal_c^1 (W_i) \right\}.
\end{align*}
Let us write $d(z):= b(\varphi_i (z)_1) c (\varphi_i (z)')$ for $z \in U_i$.

\smallskip

Let $(p_i)_{i \in I}$ be a sequence of positive numbers such that $\sum_{i \in I} p_i \Leb (W_i) = 1$. 
By Proposition~\ref{prop:MesuresAbsContinues}, without loss of generality, we replace $\tilde{\mu}$ 
by $\hat{\mu} := \sum_{i \in I} p_i \varphi_i^* \Leb_{|W_i}$.

\smallskip

Let $f_j = a_j e^{i \xi_j \cdot} d_j$, with $j \in \{1,2\}$, be in $E$. If the $d_j$ have disjoint support, 
then $\Ebb (\overline{f}_1 \cdot f_2 \circ g_t) = 0 = \Ebb ( \overline{f}_1 \Ebb (f_2| \Ical))$ for all $t$, 
and there is nothing more to prove. We assume without loss of generality that the $h_j$ are supported by 
the same open set $V_i$. Let $h (z_1) := \overline{b}_1 (z_1) b_2 (z_1) \in \Ccal_c^1 (\R_+^*)$. Then, for all $t \geq 0$:
\begin{align*}
\int_\Omega \overline{f}_1 \cdot f_2 \circ g_t \dd \mu 
& = \int_M \overline{d}_1 (z) d_2 (z) \int_A \overline{a}_1 (x) a_2 (T^{\lfloor v(x) t \rfloor} x) \int_0^1 e^{-i \xi_1 y} e^{i \xi_2 (y+v(z)t)} \dd y \dd \nu (x) \dd \hat{\mu} (z) \\
& = \delta_{\xi_1 \xi_2} \int_M \overline{d}_1 (z) d_2 (z) e^{i \xi_1 v(z)t} \int_A \overline{a}_1 (x) a_2 (T^{\lfloor v(z) t \rfloor} x) \dd \nu (x) \dd \hat{\mu} (z) \\
& = \delta_{\xi_1 \xi_2} p_i \int_{W_i} \overline{d}_1 (\varphi_i^{-1} (z)) d_2 (\varphi_i^{-1} (z)) e^{i \xi_1 v (\varphi_i^{-1}(z))t} \int_A \overline{a}_1 (x) a_2 (T^{\lfloor v (\varphi_i^{-1}(z)) t \rfloor} x) \dd \nu (x) \dd z \\
& = \delta_{\xi_1 \xi_2} p_i \int_{\R^{n-1}} \overline{c}_1 (z') c_2 (z') \dd z' \cdot \int_0^{+\infty} h (z_1) e^{i \xi_1 z_1 t} \int_A \overline{a}_1 (x) a_2 (T^{\lfloor z_1 t \rfloor} x) \dd \nu (x) \dd z_1.
\end{align*}

If $\xi_1 \neq \xi_2$, there is nothing more to prove. Assume that $\xi_1 = \xi_2=:\xi$.  Then:
\begin{align*}
\int_0^{+ \infty} h (z_1) e^{i \xi z_1 t} & \int_A \overline{a}_1 (x) a_2 (T^{\lfloor z_1 t \rfloor} x) \dd \nu (x) \dd z_1 \\
& = \int_A \overline{a}_1 (x) \int_0^ {+ \infty} e^{i \xi z_1 t} a_2 (T^{\lfloor z_1 t \rfloor} x) h (z_1) \dd z_1 \dd \nu (x) \\
& = \int_A \overline{a}_1 (x) \sum_{k=0}^{+\infty} a_2 (T^k x) \int_0^{\frac{1}{t}} e^{i \xi t s} h \left( \frac{k}{t}+s \right) \dd s \dd \nu (x).
\end{align*}
We now distinguish between two cases, depending on whether $\xi=0$ or not.

\medskip
\textsc{Case 1: $\xi \neq 0$.}
\smallskip

In the spirit of Riemann-Lebesgue's lemma, we use an integration by parts to show that the oscillations make the integral decay.
\begin{align*}
\Bigg| \int_0^{+\infty} h (z_1) e^{i \xi z_1 t} & \int_A \overline{a}_1 (x) a_2 (T^{\lfloor z_1 t \rfloor} x) \dd \nu (x) \dd z_1 \Bigg| \\
& = \frac{1}{|\xi| t} \left| \int_A \overline{a}_1 (x) \sum_{k=0}^{+\infty} a_2 (T^k x) \int_0^{\frac{1}{t}} (1-e^{i \xi t s}) h' \left( \frac{k}{t}+s \right) \dd s \dd \nu (x) \right| \\
& \leq \frac{2}{|\xi| t} \int_A |a_1| (x) \sum_{k=0}^{+\infty} |a_2| (T^k x) \int_0^{\frac{1}{t}} |h'| \left( \frac{k}{t}+s \right) \dd s \dd \nu (x) \\
& \leq \frac{2 \norm{a_1}{\Lbb^2} \norm{a_2}{\Lbb^2} \norm{h}{BV}}{|\xi| t}.
\end{align*}
By integrating over $z'$, we get:
\begin{equation*}
\left| \Ebb_\mu ( \overline{f}_1 \cdot f_2 \circ g_t ) \right| 
\leq \frac{2 \norm{a_1}{\Lbb^2} \norm{a_2}{\Lbb^2} \norm{c_1}{\Lbb^2} \norm{c_1}{\Lbb^2} \norm{h}{BV}}{|\xi| t} 
\to_{t \to + \infty} 0.
\end{equation*}
But $\Ebb_\mu (f_2|\Ical) = 0$, so the integral converges to $\Ebb_\mu (\overline{f}_1 \Ebb_\mu (f_2|\Ical))$.

\medskip
\textsc{Case 2: $\xi = 0$.}
\smallskip

In this case,
\begin{align}
\int_0^{+ \infty} h (z_1) e^{i \xi z_1 t} & \int_A \overline{a}_1 (x) a_2 (T^{\lfloor z_1 t \rfloor} x) \dd \nu (x) \dd z_1 \nonumber \\
& = \int_A \overline{a}_1 (x) \sum_{k=0}^{+\infty} a_2 (T^k x) \frac{1}{t} h \left( \frac{k}{t} \right) \dd \nu (x) + O ( t^{-1}), \label{eq:ApparitionSommeBirkhoff}
\end{align}
where:
\begin{equation*}
|O (t^{-1})| 
\leq \frac{\norm{a_1}{\Lbb^2} \norm{a_2}{\Lbb^2} \norm{h}{\Ccal^1}}{t}.
\end{equation*}

By von Neumann's ergodic theorem, 
\begin{align*}
\lim_{t \to + \infty} \sum_{k=0}^{+\infty} a_2 (T^k x) \frac{1}{t} h \left( \frac{k}{t} \right) 
& = \int_0^{+\infty} h (z_1) \dd z_1 \cdot \lim_{n \to + \infty} \frac{1}{N} \sum_{k=0}^{N-1} a_2 \circ T^k \\
& = \int_0^{+\infty} h (z_1) \dd z_1 \cdot \Ebb_\nu (a_2 | \Ical_A),
\end{align*}
where the convergence is in $\Lbb^2$ norm. Hence,
\begin{align*}
\lim_{t \to + \infty} \int_0^{+ \infty} h (z_1) e^{i \xi z_1 t} & \int_A \overline{a}_1 (x) a_2 (T^{\lfloor z_1 t \rfloor} x) \dd \nu (x) \dd z_1 \\
& = \int_0^{+\infty} h (z_1) \dd z_1 \cdot \int_A \overline{a}_1 \Ebb_\nu (a_2 | \Ical_A) \dd \nu,
\end{align*}
so that:
\begin{align*}
\lim_{t \to + \infty} \Ebb_\mu ( \overline{f}_1 \cdot f_2 \circ g_t ) 
& = p_i \int_M \overline{d}_1 d_2 \dd \varphi_i^* \Leb \cdot \int_A \overline{a}_1 \Ebb_\nu (a_2 | \Ical_A) \dd \nu \\
& = p_i \int_M \overline{d}_1 \Ebb_{\hat{\mu}} (d_2 | \Bcal_M) \dd \varphi_i^* \Leb \cdot \int_A \overline{a}_1 \Ebb_\nu (a_2 | \Ical_A) \dd \nu \\
& = \Ebb_\mu ( \overline{f}_1 \Ebb_\mu (f_2 | \Ical)). \qedhere
\end{align*}
\end{proof}

Since the sufficient criterion in Theorem~\ref{thm:CisaillementSuspensions} is the same 
as in Theorem~\ref{thm:ClassiquePrincipal}, genericity follows (as for Proposition~\ref{prop:BaireGenericite}):

\begin{corollary}\quad

Let $(A, \nu, T)$ be a system preserving a probability measure, $M$ a $n$-dimensional manifold (with $n \geq 1$). 
Let $r \in [1, + \infty]$. For $v \in \Ccal^r (M, \R_+^*)$, let $(\Omega, (g_t^v)_{t \geq 0})$ be defined as above.

\smallskip

For $\Ccal^r$ generic roof functions $v$, the system $(\Omega, \mu, (g_t^v)_{t \geq 0})$ exhibits Keplerian shear 
for any absolutely continuous compatible measure $\mu$.
\end{corollary}

%
%
%
%
%
%
%
%

We shall not discuss the speed of decay of correlations for such systems: not only do the 
critical points of $v$ matter, so do the decay of correlations on $(A, \nu, T)$.

\section{Systems without Keplerian shear}
\label{sec:SansMelange}

While systems with Keplerian shear are abundant in the classes we discussed -- since the conditions 
in Theorems~\ref{thm:ClassiquePrincipal} and~\ref{thm:CisaillementSuspensions} are generic --, 
we shall finish with a couple of examples of non-ergodic systems without Keplerian shear. 
The first is the geodesic flow on the sphere, which falls in the setting of Section~\ref{sec:FibreBundle} 
but lacks asynchronicity; the second is given by a large class of $p$-adic translations.

\subsection{Geodesic flow on a sphere}
\label{subsec:FlotTS}

Let $n \geq 2$. The manifold $T^1 \Sbb_n$ is a fibre bundle over the oriented Grassmannian 
$\widetilde{Gr} (2, n+1)$ with fibre $\Sbb_1$. This comes from the fact that the orbits of the geodesic 
flow on this manifold are oriented grand circles, and the space of oriented grand circle is isomorphic 
to the space of oriented $2$-planes in $\R^{n+1}$. The geodesic flow acts by translations on 
the grand circles. Hence the dynamical system $(T^1 \Sbb_n, \Liouv, (g_t))$ belongs to the class 
of examples discussed in Section~\ref{sec:FibreBundle}. The invariant $\sigma$-algebra $\Ical$ is 
isomorphic to $\Bcal_{\widetilde{Gr} (2, n+1)}$, and thus non trivial.

\smallskip

However, all grand circles are of the same length, so $g_{t+2\pi} = g_t$. In particular, 
given any integrable function $h$ which is not $\Ical$-measurable, the sequence of functions $(h \circ g_t)_t$ 
cannot converge to a $(g_t)$-invariant function.

\smallskip

Finally, the geodesic flow on $T^1 \Sbb_1$ is isomorphic to the disjoint union of two rotations 
on $\Sbb_1$, which are ergodic but not mixing. Hence, the system $(T^1 \Sbb_n, \Liouv, (g_t))$ 
does not have Keplerian shear for any $n \geq 1$.

\subsection{$p$-adic translations}
\label{subsec:PAdique}

Until now, we have seen classes of dynamical systems for which Keplerian shear is generic, 
with the geodesic flow on $T^1 \Sbb_n$ being an exception rather than the rule. As we shall 
see now, the situation is completely different for $p$-adic translations. 
Recall that, for $p$ a prime number, the ring $\Z_p$ is the completion of $\Z$ for the $p$-adic 
norm. It is compact, and thus supports an invariant probability, which we shall denote $\Leb$.

\smallskip

We shall see that, when one replaces translations on a torus by translations on $\Z_p$, 
the system they get typically does not exhibit Keplerian shear. The reason is that, on $\Z_p$, 
errors do not accumulate: if we change a translation on $\Z_p$ by a small quantity, 
the iterates of the two translations still stay close one to another at all times.

\begin{proposition}\quad

Let $p$ be a prime number, $d \geq 1$. Let $(M, \nu)$ be a standard probability space. 
Let $v : M \to (\Z_p)^d$ be measurable. Let:
\begin{equation*}
T : \left\{
\begin{array}{lll}
M \times (\Z_p)^d & \to & M \times (\Z_p)^d \\
(x,y) & \mapsto & (x,y+v(x))
\end{array}
\right. .
\end{equation*}
Then $(M \times \Z_p, \nu \otimes \Leb, T)$ exhibits Keplerian shear if and only if $v \equiv 0$ 
almost everywhere.
\end{proposition}

\begin{proof}

If $v \equiv 0$ almost everywhere, then $T$ is essentially the identity, which has Keplerian shear. 
Assume that this is not the case. Then one can find $A \subset M$, $N \geq 0$, $i \in \{1, \ldots, d\}$ 
and $k \in \{1, \ldots, p-1\}$ such that $\nu(A) >0$ and $v_i (x) = k p^N+\ell(x) p^{N+1}$ for all $x \in A$.

\smallskip

Let $\chi$ be a non-trivial character on $\Z_{/p\Z}$. Let 
\begin{equation*}
f : \left\{
\begin{array}{lll}
A \times (\Z_p)^d & \to & \C \\
\left(x,\left(\sum_{\ell \geq 0} y_{\ell,i} p^\ell\right)_{1 \leq i \leq d} \right) & \mapsto & \chi(y_{N,i})
\end{array}
\right. ,
\end{equation*}
Then, for $(x,y) \in A \times (\Z_p)^d$,
\begin{equation*}
f \circ T^n (x,y) 
= \chi(y_{N,i}+nk) 
= \chi(y_{N,i})\chi(k)^n. 
\end{equation*}
The function $f$ is non-zero on a set of positive measure, and since $\chi(k)$ is a non-trivial 
$p$th root of the unit, we get that $(f \circ T^n)_{n \geq 0}$ is exactly $p$-periodic. Hence, 
the system $(M \times \Z_p, \nu \otimes \Leb, T)$ does not exhibit Keplerian shear.
\end{proof}


\begin{thebibliography}{Biblio}

\bibitem{ChaikaHubert:2015} 
\newblock J.~Chaika and P.~Hubert, 
\newblock \emph{Circle averages and disjointness in typical flat surfaces on every Teichm\"uller disc}, 
\newblock arXiv:1510.05955 [math.DS], May 2017. 

\bibitem{Hormander:1983} 
\newblock L.~Hormander, 
\newblock \emph{The analysis of linear partial differential operators. I. Distribution theory and Fourier analysis}, 
\newblock Grundlehren der Mathematischen Wissenschaften, 256. Springer-Verlag, Berlin, 1983.

\bibitem{Huxley:2003} 
\newblock M.N.~Huxley, 
\newblock \emph{Exponential sums and lattice points. III.}, 
\newblock Proceedings of the London Mathematical Society (3), \textbf{87} (2003), no.~3, 591--609.

\bibitem{Jacobi:1866} 
\newblock C.~Jacobi, 
\newblock \emph{Vorlesungen \"uber Dynamik}, 
\newblock G.~Reimer, Berlin, 1866 (in German and Latin).

\bibitem{KacemLoiselMaumeDeschamps:2016} 
\newblock M.~Kacem, S.~Loisel and V.~Maume-Deschamps,
\newblock \emph{Some mixing properties of conditionally independent processes}, 
\newblock Communications in Statistics. Theory and Methods, \textbf{45} (2016), no.~5, 1241--1259.

\bibitem{Landau:1946} 
\newblock L.~Landau, 
\newblock \emph{On the vibrations of the electronic plasma}, 
\newblock Acad. Sci. USSR. J. Phys., \textbf{10} (1946), 25--34. 

\bibitem{Maucourant:2016} 
\newblock F.~Maucourant, 
\newblock \emph{Unique ergodicity of asynchronous rotations, and application}, 
\newblock arXiv:1609.04581v2 [math.DS], Jan. 2017.

\bibitem{Monteil:2005} 
\newblock T.~Monteil, 
\newblock \emph{Illumination  dans  les billards  polygonaux  et dynamique  symbolique}, 
\newblock PhD thesis, Universit\'e de la M\'editerran\'ee, 2005 (in French).

\bibitem{Moser:1978} 
\newblock J.~Moser, 
\newblock \emph{Various aspects of integrable Hamiltonian systems}, 
\newblock Dynamical systems (Bressanone, 1978), pp. 137--195, Liguori, Naples, 1980. 

\bibitem{MouhotVillani:2011} 
\newblock C.~Mouhot and C.~Villani, 
\newblock \emph{On Landau damping}, 
\newblock Acta Mathematica, \textbf{207} (2011), no.~1, 29--201. 

\bibitem{Rao:2009} 
\newblock B.L.S.~Prakasa~Rao, 
\newblock \emph{Conditional independence, conditional mixing and conditional association}, 
\newblock Annals of the Institute of Statistical Mathematics, \textbf{61} (2009), 441--460. 

\bibitem{Tabachnikov:2002} 
\newblock S.~Tabachnikov, 
\newblock \emph{Ellipsoids, complete integrability and hyperbolic geometry}, 
\newblock Moscow Mathematical Journal, \textbf{2} (2002), no.1, 183--196. 

\bibitem{Tao:2008}
\newblock T.~Tao, 
\newblock \emph{254A, Lecture 14: Weakly mixing extensions}, 
\newblock Blog post. Address: https://terrytao.wordpress.com/2008/03/02/254a-lecture-14-weakly-mixing-extensions/, 
Mar. 2, 2008. Retrieved Jan. 2018.

\bibitem{Tiscareno:2012} 
\newblock M.S.~Tiscareno, 
\newblock \emph{Planetary rings}, 
\newblock arXiv:1112.3305 [astro-ph.EP], Jul. 2012. 

\bibitem{Triebel:2006} 
\newblock H.~Triebel. 
\newblock \emph{Theory of function spaces. III}, 
\newblock Monographs in Mathematics, \textbf{100}. Birkh\"auser Verlag, Basel, 2006.
\end{thebibliography}
\end{document}